\documentclass[a4paper, 12pt]{amsart}
\RequirePackage{plautopatch}
\RequirePackage[l2tabu, orthodox]{nag}
\usepackage{bxtexlogo}
\usepackage{amsfonts}
\usepackage{bbm}
\usepackage{amssymb, amsmath,stmaryrd,mathrsfs,mathtools}
\usepackage{tikz,float}
\usepackage{tikz-cd}
\usepackage{amsthm}
\usepackage{enumerate}
\usetikzlibrary{calc}
\usepackage{url}
\usepackage[hidelinks]{hyperref}
\usepackage{graphicx}
\usepackage{xcolor}
\usepackage {diagbox}

\theoremstyle{definition}
\newtheorem{thm}{Theorem}

\newtheorem{lem}[thm]{Lemma}
\newtheorem{cor}[thm]{Corollary}
\newtheorem{exa}[thm]{Example}
\newtheorem{dfn}[thm]{Definition}
\newtheorem{rem}[thm]{Remark}
\newtheorem{que}[thm]{Question}
\newtheorem{fac}[thm]{Fact}
\newtheorem{con}[thm]{Construction}
\newtheorem{ass}[thm]{Assumption}

\newtheorem*{acknowledgements}{Acknowledgements}

\newcounter{enuAlph}

\newtheorem{teorema}[enuAlph]{Theorem}

\newtheorem{dfnfac}[thm]{Definition and Fact}

\numberwithin{thm}{section} 
\numberwithin{equation}{section}

\renewcommand{\a}{\alpha}
\renewcommand{\b}{\beta}

\newcommand{\p}{\mathbb{P}}
\newcommand{\q}{\mathbb{Q}}
\newcommand{\br}{\mathbb{R}}

\newcommand{\qd}{\dot{\mathbb{Q}}}

\DeclareMathAlphabet{\mymathbb}{U}{BOONDOX-ds}{m}{n}
\newcommand{\zero}{\mymathbb{0}}

\newcommand{\oo}{{\omega}^\omega}
\newcommand{\ooo}{[{\omega}]^\omega}

\newcommand{\sq}{2^{<\omega}}

\newcommand{\on}{\mathpunct{\upharpoonright}}

\newcommand{\bb}{\mathfrak{b}}
\newcommand{\cc}{\mathfrak{c}}
\newcommand{\dd}{\mathfrak{d}}
\newcommand{\ee}{\mathfrak{e}}
\newcommand{\pr}{\mathfrak{pr}}

\newcommand{\m}{\mathcal{M}}
\newcommand{\n}{\mathcal{N}}

\DeclareMathOperator{\non}{non}
\DeclareMathOperator{\cov}{cov}
\DeclareMathOperator{\add}{add}
\DeclareMathOperator{\cof}{cof}

\newcommand{\covm}{\cov(\mathcal{M})}
\newcommand{\nonm}{\non(\mathcal{M})}
\newcommand{\addm}{\add(\mathcal{M})}
\newcommand{\cofm}{\cof(\mathcal{M})}
\newcommand{\covn}{\cov(\mathcal{N})}
\newcommand{\nonn}{\non(\mathcal{N})}
\newcommand{\addn}{\add(\mathcal{N})}
\newcommand{\cofn}{\cof(\mathcal{N})}
\newcommand{\covi}{\cov(I)}
\newcommand{\noni}{\non(I)}
\newcommand{\addi}{\add(I)}
\newcommand{\cofi}{\cof(I)}

\newcommand{\R}{\mathbf{R}}

\renewcommand{\lq}{\preceq_T}

\DeclareMathOperator{\dom}{dom}
\DeclareMathOperator{\ran}{ran}

\DeclareMathOperator{\cf}{cf}

\newcommand{\eeb}{\mathfrak{e}^*}
\newcommand{\preb}{\mathfrak{pr}^*}

\newcommand{\prp}{\mathbb{PR}}

\newcommand{\pst}{\p}

\newcommand{\forces}{\Vdash}

\makeatletter
\renewcommand{\p@enumi}{}
\makeatother

\DeclareMathOperator{\pf}{pf}

\newcommand{\bari}{\bar{I}=\langle I_k\rangle_{k<\omega}}

\newcommand{\lebb}{\mathbb{LE}}
\newcommand{\din}{d^{-1}(\{1\})}
\newcommand{\newd}{(d^\prime)^{-1}(\{1\})\setminus d^{-1}(\{1\})}
\DeclareMathOperator{\New}{New}

\newcommand{\ufi}{\mathrm{UFI}}
\newcommand{\ufic}{\mathrm{UFIC}}

\DeclareMathOperator{\dns}{\dot{d}}

\newcommand{\lebf}{\mathbf{LE}}

\newcommand{\ind}{i}

\newcommand{\none}{\non(\mathcal{E})}
\newcommand{\cove}{\cov(\mathcal{E})}
\newcommand{\adde}{\add(\mathcal{E})}
\newcommand{\cofe}{\cof(\mathcal{E})}

\DeclareMathOperator{\Pred}{Pred}
\newcommand{\prbf}{\mathbf{PR}}
\newcommand{\inG}{\dot{L}}

\title{Cicho\'n's maximum with cardinals of the closed null ideal}%
\author{Takashi Yamazoe}
\address{Graduate School of System Informatics, Kobe University,
	Rokko--dai 1--1, Nada--ku, 657--8501 Kobe, Japan}
\email{212x502x@cloud.kobe-u.jp}
\begin{document}
	\baselineskip=17pt
	
	\setlength{\abovedisplayskip}{5pt} 
	\setlength{\belowdisplayskip}{5pt}

	\begin{abstract}
		Let $\mathcal{E}$ denote the $\sigma$-ideal generated by closed null sets on the reals. We show that the uniformity and the covering of $\mathcal{E}$ can be added to Cicho\'n's maximum with distinct values. More specifically, it is consistent that $\aleph_1<\mathrm{add}(\mathcal{N})<\mathrm{cov}(\mathcal{N})<\mathfrak{b}<\mathrm{non}(\mathcal{E})<\mathrm{non}(\mathcal{M})<\mathrm{cov}(\mathcal{M})<\mathrm{cov}(\mathcal{E})<\mathfrak{d}<\mathrm{non}(\mathcal{N})<\mathrm{cof}(\mathcal{N})<2^{\aleph_0}$ holds. 
	\end{abstract}
	
	\maketitle
	
	\section{Introduction}\label{sec_intro}
	\subsection{The ideal $\mathcal{E}$}
	Let $\mathcal{E}$ denote the $\sigma$-ideal generated by closed null sets on the reals\footnote{It is known that when dealing with cardinal invariants of $\mathcal{E}$, it does not matter in which space $\mathcal{E}$ is defined, the real line $\br$, the unit interval $[0,1]$, the Baire space $\oo$ or the Cantor space $2^\omega$ (see e.g. \cite[Remark 4.1(3)]{CM23}).\label{footnote:E}}.
	Closed null sets are nowhere dense, so $\mathcal{E}\subseteq\m\cap\n$. Moreover, 
	Bartoszy\'{n}ski and Shelah \cite{BS92} proved $\add(\mathcal{E})=\addm$ and $\cof(\mathcal{E})=\cofm$. 
	Thus, the four cardinal invariants $\add(\mathcal{E})$, $\non(\mathcal{E})$, $\cov(\mathcal{E})$ and $\cof(\mathcal{E})$ are embedded into Cicho\'n's diagram as in Figure \ref{fig_Cd_only_E}.
	
	\begin{figure}[h]
		\begin{tikzpicture}
			\tikzset{
				textnode/.style={text=black}, 
			}
			\tikzset{
				edge/.style={color=black, thin}, 
			}
			\tikzset{cross/.style={preaction={-,draw=white,line width=9pt}}}
			\newcommand{\w}{2.4}
			\newcommand{\h}{1.65}
			
			\node[textnode] (addN) at (0,  0) {$\addn$};
			\node[textnode] (covN) at (0,  \h*2) {$\covn$};
			
			\node[textnode] (nonE) at (0.6*\w,  1.4*\h) {$\none$}; 
			\node[textnode] (covE) at (2.4*\w,  0.6*\h) {$\cove$}; 
			
			
			\node[textnode] (addM) at (\w,  0) {$\addm$};
			\node[textnode] (addE) at (\w,  -0.7*\h) {$\adde$};
			\node[textnode] (b) at (\w,  \h) {$\bb$};
			\node[textnode] (nonM) at (\w,  \h*2) {$\nonm$};
			
			\node[textnode] (covM) at (\w*2,  0) {$\covm$};
			\node[textnode] (d) at (\w*2,  \h) {$\dd$};
			\node[textnode] (cofE) at (\w*2,  2.7*\h) {$\cofe$};
			\node[textnode] (cofM) at (\w*2,  \h*2) {$\cofm$};

			\node[textnode] (nonN) at (\w*3,  0) {$\nonn$};
			\node[textnode] (cofN) at (\w*3,  \h*2) { $\cofn$};
			
			\node[textnode] (aleph1) at (-\w,  0) {$\aleph_1$};
			\node[textnode] (c) at (\w*4,  \h*2) {$2^{\aleph_0}$};
			
			\draw[->, edge] (nonE) to (nonM);
			\draw[->, edge] (nonE) to (nonN);
			\draw[->, edge] (addM) to (nonE);
			
			\draw[->, edge] (covM) to (covE);
			\draw[->, edge] (covN) to (covE);
			\draw[->, edge] (covE) to (cofM);
			
			\draw[->, edge] (addN) to (covN);
			\draw[->, edge] (addN) to (addM);
			\draw[->, edge] (covN) to (nonM);	
			\draw[->, edge] (addM) to (b);
			\draw[->, edge] (b) to (nonM);
			\draw[->, edge] (addM) to (covM);
			\draw[->, edge] (nonM) to (cofM);
			\draw[->, edge] (covM) to (d);
			\draw[->, edge] (d) to (cofM);
			\draw[->, edge] (b) to (d);
			\draw[->, edge] (covM) to (nonN);
			\draw[->, edge] (cofM) to (cofN);
			\draw[->, edge] (nonN) to (cofN);
			\draw[->, edge] (aleph1) to (addN);
			\draw[->, edge] (cofN) to (c);
			
			\draw[double distance=2pt, edge] (addM) to (addE);
			\draw[double distance=2pt, edge] (cofM) to (cofE);

		\end{tikzpicture}
		\caption{Cicho\'n's diagram with the four cardinal invariants of $\mathcal{E}$.}\label{fig_Cd_only_E}
	\end{figure}

	\subsection{Cicho\'n's maximum}
	Cicho\'n's maximum means a ``maximal'' separation constellation of Cicho\'n's diagram, where all the cardinal invariants in the diagram have distinct values, except for the two dependent numbers $\cofm=\max\{\dd,\nonm\}$ and $\addm=\min\{\bb,\covm\}$.
	Goldstern, Kellner and Shelah \cite{GKS} constructed a model of Cicho\'n's maximum (assuming four strongly compact cardinals) whose separation order is as in Figure \ref{fig_CM_GKS}.
	Later, they and Mej\'{\i}a \cite{GKMS} 
	constructed models of Cicho\'n's maximum for the two constellations from \cite{GKS} and \cite{KST19}, assuming the consistency of ZFC alone. 
	\begin{figure}
		\centering
		\begin{tikzpicture}
			\tikzset{
				textnode/.style={text=black}, 
			}
			\tikzset{
				edge/.style={color=black, thin, opacity=0.5}, 
			}
			\newcommand{\w}{2.4}
			\newcommand{\h}{1.5}
			
			\node[textnode] (addN) at (0,  0) {$\addn$};
			\node (t1) [fill=lime, draw, text=black, circle,inner sep=1.0pt] at (-0.25*\w, 0.5*\h) {$\theta_1$};
			
			\node[textnode] (covN) at (0,  2*\h) {$\covn$};
			\node (t2) [fill=lime, draw, text=black, circle,inner sep=1.0pt] at (0.25*\w, 1.5*\h) {$\theta_2$};

			\node[textnode] (addM) at (\w,  0) {$\cdot$};
			\node[textnode] (b) at (\w,  1*\h) {$\bb$};
			\node (t3) [fill=lime, draw, text=black, circle,inner sep=1.0pt] at (0.75*\w, 1.2*\h) {$\theta_3$};
			
			\node[textnode] (nonM) at (\w,  2*\h) {$\nonm$};
			\node (t4) [fill=lime, draw, text=black, circle,inner sep=1.0pt] at (1.25*\w, 2.4*\h) {$\theta_4$};
			
			\node[textnode] (covM) at (\w*2,  0) {$\covm$};
			\node (t5) [fill=lime, draw, text=black, circle,inner sep=1.0pt] at (1.75*\w, -0.4*\h) {$\theta_5$};
			
			\node[textnode] (d) at (\w*2,  1*\h) {$\dd$};
			\node (t6) [fill=lime, draw, text=black, circle,inner sep=1.0pt] at (2.25*\w, 0.8*\h) {$\theta_6$};
			
			\node[textnode] (cofM) at (\w*2,  2*\h) {$\cdot$};

			\node[textnode] (nonN) at (\w*3,  0) {$\nonn$};
			\node (t7) [fill=lime, draw, text=black, circle,inner sep=1.0pt] at (2.75*\w, 0.5*\h) {$\theta_7$};
			
			\node[textnode] (cofN) at (\w*3,  2*\h) {$\cofn$};
			\node (t8) [fill=lime, draw, text=black, circle,inner sep=1.0pt] at (3.25*\w, 1.5*\h) {$\theta_{8}$};
			
			\node[textnode] (aleph1) at (-\w,  0) {$\aleph_1$};
			\node[textnode] (c) at (\w*4,  2*\h) {$2^{\aleph_0}$};
			\node (tcc) [fill=lime, draw, text=black, circle,inner sep=1.0pt] at (3.7*\w, 2.5*\h) {$\theta_\cc$};

			\draw[->, edge] (addN) to (covN);
			\draw[->, edge] (addN) to (addM);
			\draw[->, edge] (covN) to (nonM);	
			\draw[->, edge] (addM) to (b);
			
			\draw[->, edge] (addM) to (covM);
			\draw[->, edge] (nonM) to (cofM);
			
			\draw[->, edge] (d) to (cofM);
			\draw[->, edge] (b) to (d);
			\draw[->, edge] (b) to (nonM);
			\draw[->, edge] (covM) to (nonN);
			\draw[->, edge] (covM) to (d);
			\draw[->, edge] (cofM) to (cofN);
			\draw[->, edge] (nonN) to (cofN);
			\draw[->, edge] (aleph1) to (addN);
			\draw[->, edge] (cofN) to (c);

			\draw[blue,thick] (-0.5*\w,1*\h)--(0.5*\w,1*\h);
			\draw[blue,thick] (2.5*\w,1*\h)--(3.5*\w,1*\h);

			\draw[blue,thick] (0.5*\w,1.5*\h)--(1.5*\w,1.5*\h);
			\draw[blue,thick] (1.5*\w,0.5*\h)--(2.5*\w,0.5*\h);

			\draw[blue,thick] (1.5*\w,-0.5*\h)--(1.5*\w,2.5*\h);
			\draw[blue,thick] (2.5*\w,-0.5*\h)--(2.5*\w,2.5*\h);

			\draw[blue,thick] (0.5*\w,-0.5*\h)--(0.5*\w,2.5*\h);
			\draw[blue,thick] (-0.5*\w,-0.5*\h)--(-0.5*\w,2.5*\h);
			\draw[blue,thick] (3.5*\w,-0.5*\h)--(3.5*\w,2.5*\h);

		\end{tikzpicture}
		\caption{Cicho\'n's maximum constructed in \cite{GKS}. $\aleph_1<\theta_1<\cdots<\theta_8<\theta_\cc$ are regular cardinals. $\addm$ and $\cofm$ are omitted as dots ``$\cdot$'' since they have dependent values.}\label{fig_CM_GKS}
	\end{figure}
	
	Now, let us consider a maximal constellation of Figure \ref{fig_Cd_only_E}. 
	Since $\adde=\addm$ and $\cofe=\cofm$ have dependent values, 
	a (possible) maximal constellation is \textit{Cicho\'n's maximum with $\none$ and $\cove$} of distinct values.
	Due to a result of Cardona and Mej\'{\i}a, in the model of Figure \ref{fig_CM_GKS} $\none=\nonm$ and $\cove=\covm$ hold (\cite[Claim 4.11]{Car24}). 
	However, as seen in e.g. \cite[Question 6.1(3), 6.3]{Car23} and \cite[Problem 4.22 (1)]{Car24},  it was open whether there exists a model which realizes such a maximal constellation.  
	\subsection{Main result}
	We solve the question positively. That is, $\none$ and $\cove$ can be added to Cicho\'n's maximum with distinct values: 
	
	\begin{teorema}(Theorem \ref{thm_p7fin})
		\label{teo_CM_nonE}
		Let $\aleph_1\leq\theta_1\leq\cdots\leq\theta_{10}$ be regular cardinals and $\theta_\cc$ an infinite cardinal with $\theta_\cc\geq\theta_{10}$ and $\theta_\cc^{\aleph_0}=\theta_\cc$.
		Then, the separation constellation described in Figure \ref{fig_CM_nonE_covE} consistently holds.
	\end{teorema}

	
	\begin{figure}
		\centering
		\begin{tikzpicture}
			\tikzset{
				textnode/.style={text=black}, 
			}
			\tikzset{
				edge/.style={color=black, thin, opacity=0.5}, 
			}
			\newcommand{\w}{2.4}
			\newcommand{\h}{1.65}
			
			\node[textnode] (addN) at (0,  0) {$\addn$};
			\node (t1) [fill=lime, draw, text=black, circle,inner sep=1.0pt] at (-0.25*\w, 0.4*\h) {$\theta_1$};
			
			\node[textnode] (covN) at (0,  2*\h) {$\covn$};
			\node (t2) [fill=lime, draw, text=black, circle,inner sep=1.0pt] at (-0.25*\w, 1.5*\h) {$\theta_2$};

			\node[textnode] (addM) at (\w,  0) {$\cdot$};
			\node[textnode] (b) at (\w,  1*\h) {$\bb$};
			\node (t3) [fill=lime, draw, text=black, circle,inner sep=1.0pt] at (1.2*\w, 0.7*\h) {$\theta_3$};
			
			\node[textnode] (nonE) at (0.6*\w,  1.4*\h) {$~~~~~\none$};

			\node[textnode] (nonM) at (\w,  2*\h) {$\nonm$};
			\node (t5) [fill=lime, draw, text=black, circle,inner sep=1.0pt] at (1.2*\w, 2.5*\h) {$\theta_5$};
			
			\node[textnode] (covM) at (\w*2,  0) {$\covm$};
			\node (t6) [fill=lime, draw, text=black, circle,inner sep=1.0pt] at (1.8*\w, -0.5*\h) {$\theta_6$};
			
			\node[textnode] (covE) at (2.4*\w,  0.6*\h) {$\cove~~~~~$};

			\node[textnode] (d) at (\w*2,  1*\h) {$\dd$};
			\node (t8) [fill=lime, draw, text=black, circle,inner sep=1.0pt] at (1.8*\w, 1.3*\h) {$\theta_8$};
			
			\node[textnode] (cofM) at (\w*2,  2*\h) {$\cdot$};

			\node[textnode] (nonN) at (\w*3,  0) {$\nonn$};
			\node (t9) [fill=lime, draw, text=black, circle,inner sep=1.0pt] at (3.25*\w, 0.5*\h) {$\theta_9$};
			
			\node[textnode] (cofN) at (\w*3,  2*\h) {$\cofn$};
			\node (t10) [fill=lime, draw, text=black, circle,inner sep=1.0pt] at (3.25*\w, 1.6*\h) {$\theta_{10}$};
			
			\node[textnode] (aleph1) at (-\w,  0) {$\aleph_1$};
			\node[textnode] (c) at (\w*4,  2*\h) {$2^{\aleph_0}$};
			\node (tcc) [fill=lime, draw, text=black, circle,inner sep=1.0pt] at (3.7*\w, 2.5*\h) {$\theta_{\mathfrak{c}}$};

			\draw[->, edge] (addN) to (covN);
			\draw[->, edge] (addN) to (addM);
			\draw[->, edge] (covN) to (nonM);	
			\draw[->, edge] (addM) to (b);
			
			\draw[->, edge] (addM) to (covM);
			\draw[->, edge] (nonM) to (cofM);
			
			\draw[->, edge] (d) to (cofM);
			\draw[->, edge] (b) to (d);
			\draw[->, edge] (b) to (nonM);
			\draw[->, edge] (covM) to (nonN);
			\draw[->, edge] (covM) to (d);
			\draw[->, edge] (cofM) to (cofN);
			\draw[->, edge] (nonN) to (cofN);
			\draw[->, edge] (aleph1) to (addN);
			\draw[->, edge] (cofN) to (c);
			
			\draw[->, edge] (nonE) to (nonM);
			\draw[->, edge] (nonE) to (nonN);
			\draw[->, edge] (addM) to (nonE);
			
			\draw[->, edge] (covM) to (covE);
			\draw[->, edge] (covN) to (covE);
			\draw[->, edge] (covE) to (cofM);

			\draw[blue,thick] (-0.5*\w,1*\h)--(0.35*\w,1*\h);
			\draw[blue,thick] (2.65*\w,1*\h)--(3.5*\w,1*\h);

			\draw[blue,thick] (0.35*\w,1.65*\h)--(1.5*\w,1.65*\h);
			\draw[blue,thick] (1.5*\w,0.35*\h)--(2.65*\w,0.35*\h);
			
			\draw[blue,thick] (0.35*\w,1.15*\h)--(1.5*\w,1.15*\h);
			\draw[blue,thick] (1.5*\w,0.85*\h)--(2.65*\w,0.85*\h);

			\draw[blue,thick] (1.5*\w,-0.5*\h)--(1.5*\w,2.5*\h);
			\draw[blue,thick] (2.65*\w,-0.5*\h)--(2.65*\w,2.5*\h);

			\draw[blue,thick] (0.35*\w,-0.5*\h)--(0.35*\w,2.5*\h);
			\draw[blue,thick] (-0.5*\w,-0.5*\h)--(-0.5*\w,2.5*\h);
			\draw[blue,thick] (3.5*\w,-0.5*\h)--(3.5*\w,2.5*\h);

			\node (t4) [fill=lime, draw, text=black, circle,inner sep=1.0pt] at (1.2*\w, 1.4*\h) {$\theta_4$};
			\node (t7) [fill=lime, draw, text=black, circle,inner sep=1.0pt] at (1.8*\w, 0.6*\h) {$\theta_7$};

		\end{tikzpicture}
		\caption{Cicho\'n's maximum with $\none$ and $\cove$.}\label{fig_CM_nonE_covE}
	\end{figure}

	Moreover, we additionally focus on the evasion number $\ee$ and its variants:
	
	\begin{dfn}
		\begin{enumerate}
			
			\item A pair $\pi=(D,\{\pi_n:n\in D\})$ is a predictor if $D\in\ooo$ and each $\pi_n$ is a function $\pi_n\colon\omega^n\to\omega$. $\Pred$ denotes the set of all predictors.
			\item A predictor $\pi=(D,\{\pi_n:n\in D\})$ predicts $f\in\oo$ if $f(n)=\pi_n(f\on n)$ for all but finitely many $n\in D$. 
			$f$ evades $\pi$ if $\pi$ does not predict $f$.
			\item The prediction number $\pr$ and the evasion number $\ee$ are defined as follows\footnote{While the name ``prediction number'' and the notation ``$\pr$'' are not common, we use them in this paper as in \cite{Yam25}.}:
			\begin{gather*}
				\mathfrak{pr}\coloneq\min\{|\Pi|:\Pi\subseteq \Pred,\forall f\in\oo~\exists\pi\in \Pi~ \pi \text{ predicts } f\},\\
				\mathfrak{e}\coloneq\min\{|F|:F\subseteq\oo,\forall \pi\in \Pred~\exists f\in F~ f\text{ evades }\pi\}.
			\end{gather*}

			\item A predictor $\pi$ \textit{bounding-predicts} $f\in\oo$ if $f(n)\leq\pi(f\on n)$ for all but finitely many $n\in D$.
			$\preb$ and $\eeb$ denote the prediction/evasion number respectively with respect to the bounding-prediction.
			\item Let $g\in\left(\omega+1\setminus2\right)^\omega$. ( ``$\setminus2$'' is required to exclude trivial cases.)
				\textit{$g$-prediction} is the prediction where the range of functions $f$ is restricted to $\prod_{n<\omega}g(n)$ and $\pr_g$ and $\ee_g$ denote the prediction/evasion number respectively with respect to $g$-prediction.
				Namely,
				\begin{gather*}
					\mathfrak{pr}_g\coloneq\min\{|\Pi|:\Pi\subseteq \Pred,\forall f\in\textstyle{\prod_{n<\omega}g(n)}~\exists\pi\in \Pi~\pi \text{ predicts } f\},\\
					\mathfrak{e}_g\coloneq\min\{|F|:F\subseteq\textstyle{\prod_{n<\omega}g(n)},\forall \pi\in \Pred~\exists f\in F~ f\text{ evades }\pi\}.
				\end{gather*}
				Define:
				\begin{gather*}
					\mathfrak{pr}_{ubd}\coloneq\sup\left\{\mathfrak{pr}_g:g\in\left(\omega\setminus2\right)^\omega\right\},\\
					\ee_{ubd}\coloneq\min\left\{\ee_g:g\in\left(\omega\setminus2\right)^\omega\right\}.
				\end{gather*}

			\end{enumerate}
			
		\end{dfn}

		We obtain another separation model including these numbers to answer \cite[Question 5.4]{Yam25} and \cite[Problem 4.22 (4)]{Car24}:
		
		\begin{teorema}(Theorem \ref{thm_p7fin})
			\label{teo_CM_e_nonE}
			Let $\aleph_1\leq\theta_1\leq\cdots\leq\theta_{12}$ be regular cardinals and $\theta_\cc$ an infinite cardinal with $\theta_\cc\geq\theta_{12}$ and $\theta_\cc^{\aleph_0}=\theta_\cc$.
			Then, the separation constellation described in Figure \ref{fig_CM_many} consistently holds.
		\end{teorema}
		
		\begin{figure}
			\centering
			\begin{tikzpicture}
				\tikzset{
					textnode/.style={text=black}, 
				}
				\tikzset{
					edge/.style={color=black, thin, opacity=0.4}, 
				}
				\newcommand{\w}{2.4}
				\newcommand{\h}{2.0}
				
				\node[textnode] (addN) at (0,  0) {$\addn$};
				\node (t1) [fill=lime, draw, text=black, circle,inner sep=1.0pt] at (-0.25*\w, 0.8*\h) {$\theta_1$};
				
				\node[textnode] (covN) at (0,  \h*3) {$\covn$};
				\node (t2) [fill=lime, draw, text=black, circle,inner sep=1.0pt] at (0.15*\w, 3.3*\h) {$\theta_2$};

				\node[textnode] (addM) at (\w,  0) {$\cdot$};
				\node[textnode] (b) at (\w,  1.3*\h) {$\bb$};
				\node (t3) [fill=lime, draw, text=black, circle,inner sep=1.0pt] at (0.68*\w, 1.2*\h) {$\theta_3$};
				
				\node[textnode] (nonM) at (\w,  \h*3) {$\nonm$};
				\node (t6) [fill=lime, draw, text=black, circle,inner sep=1.0pt] at (1.35*\w, 3.3*\h) {$\theta_6$};
				
				\node[textnode] (covM) at (\w*2,  0) {$\covm$};
				\node (t7) [fill=lime, draw, text=black, circle,inner sep=1.0pt] at (1.65*\w, -0.3*\h) {$\theta_7$};
				
				\node[textnode] (d) at (\w*2,  1.7*\h) {$\dd$};
				\node (t10) [fill=lime, draw, text=black, circle,inner sep=1.0pt] at (2.32*\w, 1.8*\h) {$\theta_{10}$};
				\node[textnode] (cofM) at (\w*2,  \h*3) {$\cdot$};

				\node[textnode] (nonN) at (\w*3,  0) {$\nonn$};
				\node (t11) [fill=lime, draw, text=black, circle,inner sep=1.0pt] at (2.85*\w, -0.3*\h) {$\theta_{11}$};
				
				\node[textnode] (cofN) at (\w*3,  \h*3) {$\cofn$};
				\node (t12) [fill=lime, draw, text=black, circle,inner sep=1.0pt] at (3.25*\w, 2.2*\h) {$\theta_{12}$};
				
				\node[textnode] (aleph1) at (-\w,  0) {$\aleph_1$};
				\node[textnode] (c) at (\w*4,  \h*3) {$2^{\aleph_0}$};
				\node (t10) [fill=lime, draw, text=black, circle,inner sep=1.0pt] at (3.67*\w, 3.4*\h) {$\theta_\cc$};
				
				\node[textnode] (e) at (0.5*\w,  1.7*\h) {$\mathfrak{e}$}; 
				\node[textnode] (estar) at (\w,  1.7*\h) {$\ee^*$};
				\node (t4) [fill=lime, draw, text=black, circle,inner sep=1.0pt] at (0.15*\w, 1.7*\h) {$\theta_4$};

				\node[textnode] (pr) at (2.5*\w,  1.3*\h) {$\mathfrak{pr}$}; 
				\node[textnode] (prstar) at (\w*2,  1.3*\h) {$\mathfrak{pr}^*$};
				\node (t9) [fill=lime, draw, text=black, circle,inner sep=1.0pt] at (2.85*\w, 1.3*\h) {$\theta_9$};
				
				\node[textnode] (eubd) at (\w*0.2,  \h*2.1) {$\ee_{ubd}$};
				\node[textnode] (nonE) at (\w*0.75,  \h*2.55) {$\non(\mathcal{E})$};
				\node (t5) [fill=lime, draw, text=black, circle,inner sep=1.0pt] at (0.5*\w, 2.1*\h) {$\theta_5$};
				
				\node[textnode] (prubd) at (\w*2.8,  \h*0.9) {$\mathfrak{pr}_{ubd}$};
				\node[textnode] (covE) at (\w*2.25,  \h*0.45) {$\cov(\mathcal{E})$};
				\node (t8) [fill=lime, draw, text=black, circle,inner sep=1.0pt] at (2.47*\w, 0.9*\h) {$\theta_8$};

				\draw[->, edge] (addN) to (covN);
				\draw[->, edge] (addN) to (addM);
				\draw[->, edge] (covN) to (nonM);	
				\draw[->, edge] (addM) to (b);
				\draw[->, edge] (addM) to (covM);
				\draw[->, edge] (nonM) to (cofM);
				\draw[->, edge] (d) to (cofM);
				\draw[->, edge] (b) to (prstar);
				\draw[->, edge] (covM) to (nonN);
				\draw[->, edge] (cofM) to (cofN);
				\draw[->, edge] (nonN) to (cofN);
				\draw[->, edge] (aleph1) to (addN);
				\draw[->, edge] (cofN) to (c);
				
				\draw[->, edge] (e) to (covM);
				\draw[->, edge] (addN) to (e);
				
				\draw[->, edge] (covM) to (prstar);
				\draw[->, edge] (nonM) to (pr);
				\draw[->, edge] (pr) to (cofN);
				
				\draw[->, edge] (e) to (estar);
				\draw[->, edge] (b) to (estar);C
				\draw[->, edge] (estar) to (nonM);
				\draw[->, edge] (estar) to (d);
				\draw[->, edge] (e) to (eubd);
				
				\draw[->, edge] (prstar) to (d);
				\draw[->, edge] (prstar) to (pr);
				
				\draw[->, edge] (prstar) to (pr);
				\draw[->, edge] (prubd) to (pr);
				\
				
				\draw[->, edge] (eubd) to (nonE);
				\draw[->, edge] (addM) to (nonE);
				\draw[->, edge] (nonE) to (nonM);
				
				\draw[->, edge] (covE) to (prubd);
				\draw[->, edge] (covE) to (cofM);
				\draw[->, edge] (covM) to (covE);
				
				\draw[blue,thick] (-0.5*\w,1.5*\h)--(3.5*\w,1.5*\h);
				\draw[blue,thick] (1.5*\w,-0.5*\h)--(1.5*\w,3.5*\h);
				
				\draw[blue,thick] (-0.5*\w,-0.5*\h)--(-0.5*\w,3.5*\h);
				\draw[blue,thick] (3.5*\w,-0.5*\h)--(3.5*\w,3.5*\h);
				
				\draw[blue,thick] (0.5*\w,-0.5*\h)--(0.5*\w,1.5*\h);
				\draw[blue,thick] (2.5*\w,1.5*\h)--(2.5*\w,3.5*\h);
				
				\draw[blue,thick] (-0.1*\w,1.9*\h)--(1.5*\w,1.9*\h);
				\draw[blue,thick] (-0.1*\w,2.7*\h)--(1.5*\w,2.7*\h);
				\draw[blue,thick] (-0.1*\w,1.5*\h)--(-0.1*\w,2.7*\h);
				\draw[blue,thick] (0.5*\w,2.7*\h)--(0.5*\w,3.5*\h);
				
				\draw[blue,thick] (3.1*\w,1.1*\h)--(1.5*\w,1.1*\h);
				\draw[blue,thick] (3.1*\w,0.3*\h)--(1.5*\w,0.3*\h);
				\draw[blue,thick] (3.1*\w,1.5*\h)--(3.1*\w,0.3*\h);
				\draw[blue,thick] (2.5*\w,0.3*\h)--(2.5*\w,-0.5*\h);

				\draw[->, edge] (nonE) to (nonN);
				\draw[->, edge] (covN) to (covE);

			\end{tikzpicture}
			\caption{Cicho\'n's maximum with $\ee$, $\ee^*$, $\ee_{ubd}$ , $\none$ and their duals.}\label{fig_CM_many}
		\end{figure}

		\subsection{Methods}
		The construction of Cicho\'n's maximum consists of two steps: first we separate the left side of the diagram and then separate the right. The second step is not the main part of this paper, so let us focus on the first one. The main work of this step is to keep the bounding number $\bb$ small through the forcing iteration, since the other cardinal invariants are kept small without special care.
		To control $\bb$, in \cite{GKS} they used the \textit{ultrafilter-limit} (abbreviated as UF-limit in this paper) method, which was first introduced by Goldstern, Mej\'{\i}a and Shelah in \cite{GMS16} who showed  (see also Table \ref{table_limits}):
		\begin{thm}(\cite[Main Lemma 4.6]{GMS16})
			UF-limits keep $\bb$ small.
		\end{thm} 
		Roughly speaking, by using this method we can take a \textit{limit condition} of a sequence of countably many conditions and the limit condition tells us some information of the original sequence and consequently makes it possible to argue the smallness of $\bb$.
		
		There are several other kinds of limit methods: Kellner, Shelah and T{\u{a}}nasie \cite{KST19} introduced the \textit{FAM-limit} method (focusing on finitely additive measures on $\omega$), which stems from Shelah's work \cite{She00} to force $\covn$ to have countable cofinality, and consequently they constructed another model of Cicho\'n's maximum with a different order.
		The author introduced in \cite{Yam25} a new limit notion called \textit{closed-UF-limit}, which is a specific kind of the original UF-limit, and proved (note $\ee\leq\ee^*$):
		\begin{thm}(\cite[Main Lemma 3.26]{Yam25})
			Closed-UF-limits keep $\ee^*$ small.
		\end{thm}
		Moreover, he \textit{mixed} two limit methods: he performed a forcing iteration which has both UF-limits and closed-UF-limits with different sizes, controlled both $\bb$ and $\ee$ at the same time, and forced them to take distinct values according to the sizes of the two limits. Finally he constructed a model of \textit{Cicho\'n's maximum with evasion number}.
		
		Thus, one may naturally consider if we can mix UF-limits and FAM-limits and actually this question has been considered recently. In 2021, Goldstern, Kellner, Mej\'{\i}a and Shelah in \cite{GKMS21} found that FAM-limits keep $\ee$ small and considered mixing UF-limits and FAM-limits to obtain Cicho\'n's maximum with $\ee$, though they found a gap in the mixing argument later. In 2023, Cardona and Mej\'{\i}a improved this result as follows (note $\ee\leq\none$): 
		\begin{thm}(\cite[Theorem 10.4]{CMU24})
			\label{thm_fam_keep_none_small}
			FAM-limits keep $\none$ small.
		\end{thm}
		
		Therefore, it turned out that if the mix is possible, we can obtain Cicho\'n's maximum with $\none$ and $\cove$, but the possibility itself was still unclear.

		%
		
		To achieve the mix, we need a forcing notion as an iterand in the iteration, satisfying the following four properties:
		\begin{enumerate}
			\item It increases $\nonm$. \label{item_prop_1}
			\item It has UF-limits.\label{item_prop_2}
			\item It has FAM-limits. \label{item_prop_3}
		\end{enumerate}
		And for a technical reason on the construction of the forcing iteration (see Lemma \ref{lem_uf_const_lim} and also Remark \ref{rem_key_centered}), 
		\begin{enumerate}
			\setcounter{enumi}{3}
			\item It is \textit{$\sigma$-centered}. \label{item_prop_4}
		\end{enumerate}
		However, previously known posets do not satisfy all the properties, though there were several close candidates. 
		For example, the standard eventually different forcing $\mathbb{E}$ satisfies \eqref{item_prop_1},\eqref{item_prop_2},\eqref{item_prop_4}, but not \eqref{item_prop_3} (see e.g. \cite[Remark 6.12]{CMU24}). 
		On the other hand, $\tilde{\mathbb{E}}$, which is a variant of $\mathbb{E}$ introduced in \cite{KST19}, satisfies \eqref{item_prop_1},\eqref{item_prop_2},\eqref{item_prop_3}, but not \eqref{item_prop_4}.
		Moreover, what made it practically difficult to tackle the problem of the mix was that the theory of FAM-limits was so complicated that there were only two known (non-trivial) posets with FAM-limits, $\tilde{\mathbb{E}}$ and random forcing $\mathbb{B}$.
		
		In 2024, Mej\'{\i}a ``anatomized'' the structure of $\tilde{\mathbb{E}}$ in \cite{Mej24} and introduced a new limit notion \textit{UF-limit for intervals} (abbreviated as UFI-limit in this paper). He distilled from the theory of FAM-limit several essential properties sufficient for controlling $\none$, and translated them into the context of UF-limits. As a result, he defined UFI-limit and showed $\tilde{\mathbb{E}}$ has UFI-limits. Later, he, Cardona and Uribe \cite{CMU24} developed a general theory of FAM-limits including UF-limits by seeing an ultrafilter as a finitely additive measure with values in $\{0,1\}$ and consequently it turned out that the following theorem holds as a special case of Theorem \ref{thm_fam_keep_none_small}: 
		%
		\begin{thm}(\cite[Theorem 10.4]{CMU24}, Theorem \ref{thm_UFI_keeps_nonE_small})
			UFI-limits keep $\none$ small.
		\end{thm}
		
		Table \ref{table_limits} summarizes the relationship between the four limit methods and the corresponding cardinal invariants kept small by the methods:
		\begin{table}[h]
			\caption{Limit methods and corresponding numbers.}\label{table_limits}
			\begin{tabular}{c|c}
				Limit method &  keep small \\
				
				\hline
				
				UF-limit & $\bb$ \\
				closed-UF-limit   & $\ee^*$ \\
				FAM-limit    & $\none$ \\
				UFI-limit    & $\none$ \\
				\hline 
				
			\end{tabular}
		\end{table}

		Thus, the third item \eqref{item_prop_3} of the previous four properties can be replaced by UFI-limits (see also Table \ref{table_required_properties}). 
		Therefore, it became easier to find the desirable forcing notion for controlling $\none$, since UF-limit, of which UFI-limit is a variant (and even a generalization by considering the interval partition of $\omega$ into singletons), is much simpler than FAM-limit, 
		and evidently there were many known forcing notions with UF-limits (e.g., $\mathbb{E}$, $\tilde{\mathbb{E}}$, prediction forcing $\prp$ (see Definition \ref{dfn_PR}), etc). 
		We finally find a new forcing notion $\lebb$ satisfying all the replaced four properties. Table \ref{table_required_properties} illustrates the relationship between the four required properties and the three forcing notions, the two close candidates $\mathbb{E}$ and $\tilde{\mathbb{E}}$ and our new $\mathbb{LE}$:
		\begin{table}[h]
			\caption{Forcing notions and required properties.}\label{table_required_properties}
			\centering
			\begin{tabular}{c||ccc}
				properties &  $\mathbb{E}$  & $\tilde{\mathbb{E}}$ & $\mathbb{LE}$ \\
				
				\hline
				
				increase $\nonm$ & $\checkmark$ & $\checkmark$ &\textcolor{black}{$\checkmark$}\\
				UF-limit   & $\checkmark$  & $\checkmark$ &  \textcolor{black}{$\checkmark$} \\
				UFI-limit   & $\times$  & $\checkmark$ &  \textcolor{black}{$\checkmark$} \\ 
				$\sigma$-centered  & $\checkmark$  & $\times$ & \textcolor{black}{$\checkmark$} \\
				
				\hline 
				
			\end{tabular}
			
		\end{table}
		
		As a result, we obtain a model of Cicho\'n's maximum with $\none$ and $\cove$.
		
		
		

		\subsection{Structure of the paper}
		In Section \ref{sec_RS_PT} and \ref{sec_UF}, we review the general theory of relational systems and UF-limits respectively, following \cite[Section 2,3]{Yam25}. In Section \ref{sec_UFI}, first we present the theory of UFI-limits, which was originally introduced in \cite{Mej24} and formalized in \cite{CMU24}. Then, we introduce the forcing notion $\lebb$ satisfying all the required properties and this is the essential main contribution of this paper. In Section \ref{sec_separation}, we construct a model of Cicho\'n's maximum with $\none$, $\cove$ and the evasion numbers to prove Theorem \ref{teo_CM_nonE} and \ref{teo_CM_e_nonE}. Finally, we conclude this paper leaving some open questions presented in Section \ref{sec_question}.
		
		\section{Review of Relational systems} \label{sec_RS_PT} 
		In this section, we list the necessary items on relational systems and the preservation theory that will be used later, without an explanation of why they are necessary and how they will be used (see \cite[Section 2]{Yam25}).
		\begin{dfn}
			
			\begin{itemize}
				\item $\R=\langle X,Y,\sqsubset\rangle$ is a relational system if $X$ and $Y$ are non-empty sets and $\sqsubset \subseteq X\times Y$.
				\item We call an element of $X$ a \textit{challenge}, an element of $Y$ a \textit{response}, and ``$x\sqsubset y$''  ``$x$ is \textit{met by }$y$''.
				\item $F\subseteq X$ is $\R$-unbounded if no response meets all challenges in $F$.
				\item $F\subseteq Y$ is $\R$-dominating if every challenge is met by some response in $F$.
				\item $\R$ is non-trivial if $X$ is $\R$-unbounded and $Y$ is $\R$-dominating. For non-trivial $\R$, define
				\begin{itemize}
					\item $\bb(\R)\coloneq\min\{|F|:F\subseteq X \text{ is }\R\text{-unbounded}\}$, and
					\item $\dd(\R)\coloneq\min\{|F|:F\subseteq Y \text{ is }\R\text{-dominating}\}$.
				\end{itemize}

			\end{itemize}
			
		\end{dfn}
		In this section, we assume $\R$ is non-trivial.
		
		\begin{dfnfac}
			\label{dfnfac_RS}
			\begin{enumerate}
				\item For $\mathbf{D}\coloneq\langle \oo,\oo,\leq^*\rangle$, we get $\bb(\mathbf{D})=\bb, \dd(\mathbf{D})=\dd$.
				\item 
				Define $\mathbf{PR}\coloneq\langle\oo, \Pred,\sqsubset^\mathrm{p}\rangle$, where $f\sqsubset^\mathrm{p}\pi:\Leftrightarrow f$ is predicted by $\pi$.
				Also, define $\mathbf{BPR}\coloneq\langle\oo, \Pred,\sqsubset^\mathrm{bp}\rangle$, where $f\sqsubset^\mathrm{bp}\pi:\Leftrightarrow f$ is bounding-predicted by $\pi$ and $\mathbf{PR}_g\coloneq\langle\prod_{n<\omega}g(n), \Pred,\sqsubset^\mathrm{p}\rangle$  where $g\in(\omega+1\setminus2)^\omega$. We have $\bb(\mathbf{PR})=\ee, \dd(\mathbf{PR})=\pr$, $\bb(\mathbf{BPR})=\eeb, \dd(\mathbf{BPR})=\preb$, $\bb(\mathbf{PR}_g)=\ee_g, \dd(\mathbf{PR}_g)=\pr_g$.
				\item\label{item_dfnfac_RS_ideal}
				For an ideal $I$ on $X$ with $[X]^{<\omega}\subseteq I$, define two relational systems $\bar{I}\coloneq\langle I,I,\subseteq\rangle$ and $C_I\coloneq\langle X,I,\in\rangle$.
				We have $\bb(\bar{I})=\addi,\dd(\bar{I})=\cofi$ and $\bb(C_I)=\noni,~\dd(C_I)=\covi$.
				If $I$ is an ideal, then we will write $\R\lq I$ to mean $\R\lq\bar{I}$; and analogously for $\succeq_T$ and $\cong_T$ (see Definition 2.4).
			\end{enumerate}
			
		\end{dfnfac}
		
		\begin{dfn}
			$\R^\bot$ denotes the dual of $\R=\langle X,Y,\sqsubset\rangle$,
			i.e.,
			$\R^\bot\coloneq\langle Y,X,\sqsubset^\bot\rangle$ where $y\sqsubset^\bot x:\Leftrightarrow \lnot(x\sqsubset y)$.
		\end{dfn}

		\begin{dfn}
			
			For relational systems $\R=\langle X,Y,\sqsubset \rangle, \R^{\prime}=\langle X^{\prime},Y^{\prime},\sqsubset^{\prime}~\rangle$,
			$(\Phi_-,\Phi_+):\R\rightarrow\R^\prime$ is a Tukey connection from $\R$ into $\R^{\prime}$ if $\Phi_-:X\rightarrow X^{\prime}$ and $\Phi_+:Y^{\prime}\rightarrow Y$ are functions such that:
			\begin{equation*}
				\forall x\in X~\forall y^{\prime}\in Y^{\prime}~\Phi_-(x)\sqsubset^{\prime} y^{\prime}\Rightarrow x \sqsubset \Phi_{+} (y^{\prime}).
			\end{equation*}

			We write $\R\preceq_T\R^{\prime}$ if there is a Tukey connection from $\R$ into $\R^{\prime}$ and call $\preceq_T$ the Tukey order.
			Tukey equivalence $\R\cong_T\R^{\prime}$ is defined as: $\R\preceq_T\R^{\prime}$ and $\R^{\prime}\preceq_T\R$. 
			
		\end{dfn}
		
		\begin{fac}
			\label{Tukey order and b and d}
			\begin{enumerate}
				\item $\R\preceq_T\R^{\prime}$ implies $(\R^{\prime})^\bot\preceq_T\R^\bot$.
				\item $\R\preceq_T\R^{\prime}$ implies $\mathfrak{b}(\R^{\prime})\leq\mathfrak{b}(\R)$ and $\mathfrak{d}(\R)\leq\mathfrak{d}(\R^{\prime})$.
				\item $\bb(\R^\bot)=\dd(\R)$ and $\dd(\R^\bot)=\bb(\R^\bot)$.
			\end{enumerate}
		\end{fac}
		In the rest of this section, we fix an uncountable regular cardinal $\theta$ and a set $A$ of size $\geq\theta$.
		
		\begin{cor}
			\textcolor{white}{a}
			\begin{itemize}
				\item If $\R\lq C_{[A]^{<\theta}}$, then $\theta\leq\bb(\R)$ and $\dd(\R)\leq|A|$.
				\item If $C_{[A]^{<\theta}}\lq \R$, then $\bb(\R)\leq\theta$ and $|A|\leq\dd(\R)$.
			\end{itemize}
		\end{cor}
		
		\begin{fac}(\cite[Lemma 1.16.]{forcing_constellations})
			\label{fac_Tukey_order_equivalence_condition}
			Assume $|X|\geq\theta$ where $\R=\langle X,Y,\sqsubset\rangle$.
			\begin{enumerate}
				\item $\R\lq C_{[X]^{<\theta}}$ iff $\bb(\R)\geq\theta$. 
				\item \label{item_small_equiv}
				$C_{[A]^{<\theta}}\lq\R$ iff there exists $\langle x_a:a\in A\rangle$ such that every $y\in Y$ meets only $<\theta$-many $x_a$.
			\end{enumerate}
		\end{fac}
		
		\begin{fac}(\cite[Lemma 1.15.]{forcing_constellations}, \cite[Fact 3.8]{CM24})
			\label{fac_suff_eq_CI_and_I}
			If $|A|^{<\theta}=|A|$,
			then $C_{[A]^{<\theta}}\cong_T [A]^{<\theta}$.
		\end{fac}

		\begin{fac}(\cite[Lemma 2.11.]{forcing_constellations})
			\label{fac_cap_V}
			Every ccc poset forces $[A]^{<\theta}\cong_T[A]^{<\theta}\cap V$ and $C_{[A]^{<\theta}}\cong_T C_{[A]^{<\theta}}\cap V$.
			Moreover, $\mathfrak{x}([A]^{<\theta})=\mathfrak{x}^V([A]^{<\theta})$ where $\mathfrak{x}$ represents ``$\add$'', ``$\cov$'', ``$\non$'' or ``$\cof$''.
		\end{fac}

		%
		\begin{fac}(\cite[Corollary 2.16, also Example 2.15]{Yam25})
			\label{fac_smallness_for_addn_and_covn_and_nonm}
			Let $\p$ be a fsi of ccc forcings of length $\gamma\geq\theta$.
			
			\begin{enumerate}
				\item Assume that each iterand is either:
				\begin{itemize}
					\item of size $<\theta$,
					\item a subalgebra of random forcing, or
					\item $\sigma$-centered.
				\end{itemize}
				Then, $\p$ forces $C_{[\gamma]^{<\theta}}\lq\n$,
				in particular, $\addn\leq\theta$.
				\item  Assume that each iterand is either:
				\begin{itemize}
					\item of size $<\theta$, or
					\item $\sigma$-centered.
				\end{itemize}
				Then, $\p$ forces $C_{[\gamma]^{<\theta}}\lq C_\n^\bot$,
				in particular, $\covn\leq\theta$.
				\item Assume that each iterand is:
				\begin{itemize}
					\item of size $<\theta$.
				\end{itemize}
				Then, $\p$ forces $C_{[\gamma]^{<\theta}}\lq C_\m$,
				in particular, $\nonm\leq\theta$.
			\end{enumerate}

		\end{fac}

		\section{Review of standard UF-limits}\label{sec_UF}
		This section is a review of \cite[Section 3]{Yam25} and we include it for the sake of completeness and in order to clarify the notation.
		\begin{dfn}(\cite[Section 5]{Mej19})
			Let $\Gamma$ be a class for subsets of posets,
			i.e., $\Gamma\in\prod_{\p}\mathcal{P}(\mathcal{P}(\p))$, a (class) function. (E.g., $\Gamma=\Lambda({\text{centered}})\coloneq$ ``centered'' is an example of a class for subsets of posets and in this case $\Gamma(\p)$ denotes the set of all centered subsets of $\p$ for each poset $\p$.)
			\begin{itemize}

				\item A poset $\p$ is $\mu$-$\Gamma$-covered if $\p$ is a union of $\leq\mu$-many subsets in $\Gamma(\p)$.
				As usual, when $\mu=\aleph_0$, we use ``$\sigma$-$\Gamma$-covered'' instead of ``$\aleph_0$-$\Gamma$-covered''. Moreover, we often just say ``$\mu$-$\Gamma$'' instead of ``$\mu$-$\Gamma$-covered''.
				\item Abusing notation, we write ``$\Gamma\subseteq\Gamma^\prime$'' if $\Gamma(\p)\subseteq\Gamma^\prime(\p)$ holds for every poset $\p$.
			\end{itemize}
		\end{dfn}
		

		In this paper, an ``ultrafilter'' means a non-principal ultrafilter on $\omega$.
		\begin{dfn}
			\label{dfn_UF_linked} 
			Let $D$ be an ultrafilter and $\p$ be a poset.
			\begin{enumerate}
				\item $Q\subseteq \p$ is $D$-lim-linked ($\in\Lambda^\mathrm{lim}_D(\p)$) if there exist a $\p$-name $\dot{D}^\prime$ of an ultrafilter extending $D$ and a function $\lim^D\colon Q^\omega\to\p$ such that for any countable sequence 
				$\bar{q}=\langle q_m:m<\omega\rangle\in Q^\omega$, 
				\begin{equation}
					\textstyle{\lim^D\bar{q}} \Vdash \{m<\omega:q_m \in \dot{G}\}\in \dot{D}^\prime.
				\end{equation}

				Moreover, if $\ran(\lim^D)\subseteq Q$, we say $Q$ is c-$D$-lim-linked (closed-$D$-lim-linked, $\in\Lambda^\mathrm{lim}_{\mathrm{c}D}(\p)$).
				\item $Q$ is (c-)uf-lim-linked (short for (closed-)ultrafilter-limit-linked) if $Q$ is (c-)$D$-lim-linked for every ultrafilter $D$.
				\item $\Lambda^\mathrm{lim}_\mathrm{uf}\coloneq\bigcap_D\Lambda^\mathrm{lim}_D$ and 
				$\Lambda^\mathrm{lim}_\mathrm{cuf}\coloneq\bigcap_D\Lambda^\mathrm{lim}_{\mathrm{c}D}$.

				
			\end{enumerate}

			We often say ``$\p$ has (c-)uf-limits'' instead of ``$\p$ is $\sigma$-(c-)uf-lim-linked''.

		\end{dfn}

		\begin{exa}
			\label{exa_size_linked}
			Singletons are c-uf-lim-linked and hence every poset $\p$ is $|\p|$-c-uf-lim-linked. 
		\end{exa}


		\begin{dfn}(\cite[Definition 4.11]{Mej24Vienna})
			\label{dfn_Gamma_iteration}
			\begin{itemize}
				\item A $\kappa$-$\Gamma$-iteration is a fsi $\langle(\p_\eta,\qd_\xi):\eta\leq\gamma,\xi<\gamma\rangle $ of ccc forcings, with witnesses $\langle\p_\xi^-:\xi<\gamma\rangle$, $\langle\theta_\xi:\xi<\gamma\rangle$ and $\langle\dot{Q}_{\xi,\zeta}:\zeta<\theta_\xi,\xi<\gamma\rangle$ satisfying for all $\xi<\gamma$:
				\begin{enumerate}
					\item $\p^-_\xi\lessdot\p_\xi$.
					\item $\theta_\xi$ is a cardinal of size $<\kappa$.
					\item \label{item_Q_is_P_minus_name}
					$\qd_\xi$ and $\langle\dot{Q}_{\xi,\zeta}:\zeta<\theta_\xi\rangle$ are $\p^-_\xi$-names and $\p^-_\xi$ forces that 
					$\bigcup_{\zeta<\theta_\xi}\dot{Q}_{\xi,\zeta}=\qd_\xi$ and $\dot{Q}_{\xi,\zeta}\in\Gamma(\qd_\xi)$ for each $\zeta<\theta_\xi$.
				\end{enumerate}
				\item $\xi<\gamma$ is a trivial stage if $\Vdash_{\p^-_\xi}|\dot{Q}_{\xi,\zeta}|=1$ for all $\zeta<\theta_\xi$. $S^-$ is the set of all trivial stages and $S^+\coloneq\gamma\setminus S^-$.
				\item A guardrail for the iteration is a function $h\in\prod_{\xi<\gamma}\theta_\xi$.
				\item $H\subseteq\prod_{\xi<\gamma}\theta_\xi$ is complete if any countable partial function in $\prod_{\xi<\gamma}\theta_\xi$ is extended to some (total) function in $H$.
				\item $\p^h_\eta$ is the set of conditions $p\in\p_\eta$ following $h$, i.e., for each $\xi\in\dom(p)$, $p(\xi)$ is a $\p^-_\xi$-name 
				and $\Vdash_{\p^-_\xi}p(\xi)\in \dot{Q}_{\xi,h(\xi)}$.  	
			\end{itemize}
		\end{dfn}
		
		%
		%
		%
		
		
		\begin{lem}(\cite[Corollary 3.8]{Yam25},\cite{EK65})
			\label{lem_complete}
			Assume $\aleph_1\leq\mu\leq|\gamma|\leq2^\mu$ and $\mu^+=\kappa$. 
			Then, for any $\langle\theta_\xi<\kappa:\xi<\gamma\rangle$, there exists a complete set of guardrails  of size $\leq\mu^{\aleph_0}$ which works for each $\kappa$-$\Gamma$-iteration of length $\gamma$ using $\langle\theta_\xi:\xi<\gamma\rangle$.
		\end{lem}

		In this section, let $\Gamma_\mathrm{uf}$ represent $\Lambda^\mathrm{lim}_\mathrm{uf}$ or $\Lambda^\mathrm{lim}_\mathrm{cuf}$.
		
		\begin{dfn}(\cite[Definition 4.15]{Mej24Vienna})
			\label{dfn_UF_iteration}
			A $\kappa$-$\Gamma_\mathrm{uf}$-iteration has $\Gamma_\mathrm{uf}$-limits on $H$ if
			\begin{enumerate}
				\item $H\subseteq\prod_{\xi<\gamma}\theta_\xi$ is a set of guardrails.
				\item For $h\in H$, $\langle \dot{D}^h_\xi:\xi\leq\gamma \rangle$ is a sequence such that $\dot{D}^h_\xi$ is a $\p_\xi$-name of an ultrafilter. 
				\item If $\xi<\eta\leq\gamma$, then $\Vdash_{\p_\eta}\dot{D}^h_\xi\subseteq\dot{D}^h_\eta$.
				
				\item \label{item_D^-}
				For $ \xi\in S^+$, $\Vdash_{\p_\xi} (\dot{D}^h_\xi)^-\in V^{\p^-_\xi}$ where $(\dot{D}^h_\xi)^-\coloneq\dot{D}^h_\xi\cap V^{\p^-_\xi}$.
				\item Whenever $\langle \xi_m:m<\omega \rangle\subseteq \gamma$ and $\bar{q}=\langle \dot{q}_m:m<\omega\rangle$ satisfying 
				
				$\Vdash_{\p^-_{\xi_m}}\dot{q}_m\in\dot{Q}_{\xi_m,h(\xi_m)}$ for each $m<\omega$:
				\begin{enumerate}
					\item If $\langle \xi_m:m<\omega \rangle$ is constant with value $\xi$, then
					\begin{equation}
						\label{eq_constant}
						\Vdash_{\p_\xi}\textstyle{\lim^{(\dot{D}^h_\xi)^-}}\bar{q}\Vdash_{\qd_\xi}\{m<\omega:\dot{q}_m\in \dot{H}_\xi\}
						\in\dot{D}^h_{\xi+1}.
					\end{equation}
					($\dot{H}_\xi$ denotes the canonical name of $\qd_\xi$-generic filter over $V^{\p_\xi}$ and abusing notation, for $\xi\in S^-$ we use $\lim^{(\dot{D}^h_\xi)^-}\bar{q}$ to denote the constant value of $\bar{q}$.)
					\item If $\langle \xi_m:m<\omega \rangle$ is strictly increasing, then
					\begin{equation}
						\label{eq_increasing}
						\Vdash_{\p_\gamma}\{m<\omega:\dot{q}_m\in \dot{G}_\gamma\}\in\dot{D}^h_\gamma.
					\end{equation}
				\end{enumerate}
				
			\end{enumerate}
		\end{dfn}

		\begin{lem}(\cite[Lemma 3.11]{Yam25}, \cite[Theorem 4.18]{Mej24Vienna})
			\label{lem_uf_const_succ}
			Let $\p_{\gamma+1}$ be a $\kappa$-$\Gamma_\mathrm{uf}$-iteration (of length $\gamma+1$) and suppose $\p_\gamma=\p_{\gamma+1}\on\gamma$ has $\Gamma_\mathrm{uf}$-limits on $H$.
			If $\gamma\in S^-$, or if $\gamma\in S^+$ and:
			\begin{equation}
				\label{eq_minus}
				\Vdash_{\p_\gamma} (\dot{D}^h_\gamma)^-\in V^{\p^-_\gamma}\text{ for all }h\in H,
			\end{equation}
			then we can find $\{\dot{D}^h_{\gamma+1}:h\in H\}$ witnessing that $\p_{\gamma+1}$ has $\Gamma_\mathrm{uf}$-limits on $H$.
		\end{lem}

		\begin{lem}(\cite[Lemma 3.13]{Yam25},\cite[Theorem 4.19]{Mej24Vienna})
			\label{lem_uf_const_lim}
			Let $\gamma$ be limit and 
			$\p_\gamma$ be a $\kappa$-$\left(\Lambda(\mathrm{centered})\cap\Gamma_\mathrm{uf}\right) $-iteration.
			If $\langle \dot{D}^h_\xi:\xi<\gamma, h\in H\rangle$ witnesses that for any $\xi<\gamma$, $\p_\xi=\p_\gamma\on\xi$ has $\Gamma_\mathrm{uf}$-limits on $H$,  
			then we can find $\langle\dot{D}^h_\gamma:h\in H\rangle$ such that $\langle \dot{D}^h_\xi:\xi\leq\gamma, h\in H\rangle$ witnesses $\p_\gamma$ has $\Gamma_\mathrm{uf}$-limits on $H$.
		\end{lem}

		\begin{ass}
			\label{ass_first_Cohen}
			\begin{enumerate}
				\item $\kappa\leq \lambda$ are uncountable cardinals, $\kappa$ is regular and $\gamma=\lambda+\lambda$.
				\item $\p_\gamma$ is a $\kappa$-$\Gamma_\mathrm{uf}$-iteration with $\Gamma_\mathrm{uf}$-limits on $H$ (with the same parameters as in Definition \ref{dfn_Gamma_iteration} and \ref{dfn_UF_iteration}).
				\item $H$ is complete and $|H|<\kappa$.
				\item For $\xi<\lambda$, $\Vdash_{\p^-_\xi}\qd_\xi=\mathbb{C}$, the Cohen forcing. Note that $\mathbb{C}$ is $\kappa$-$\Gamma_\mathrm{uf}$-linked by Example \ref{exa_size_linked}.
			\end{enumerate}
		\end{ass}
		
		\begin{thm}(\cite[Lemma 1.31]{GKS},\cite[Theorem 3.21]{Yam25},)
			\label{thm_uf_limit_keeps_b_small}
			Suppose Assumption \ref{ass_first_Cohen} and consider the case $\Gamma_\mathrm{uf}=\Lambda^\mathrm{lim}_\mathrm{uf}$.
			Then, $\p_\gamma$ forces $C_{[\lambda]^{<\kappa}}\lq \mathbf{D}$, in particular, $\bb\leq\kappa$. 
		\end{thm}

		The author showed in \cite{Yam25} that cUF-limits keep $\eeb$ small.
		\begin{thm}(\cite[Main Lemma 3.26]{Yam25})
			\label{thm_cUF_keep_estar_small}
			Suppose Assumption \ref{ass_first_Cohen} and consider the case $\Gamma_\mathrm{uf}=\Lambda^\mathrm{lim}_\mathrm{cuf}$.
			%
			Then, $\p_\gamma$ forces $C_{[\lambda]^{<\kappa}}\lq \mathbf{BPR}$, in particular, $\ee\leq\eeb\leq\kappa$.
		\end{thm}

		The author introduced in \cite{Yam25} a forcing-free characterization of ``$Q\subseteq\p$ is (c-)uf-lim-linked'':
		
		\begin{lem}(\cite[Lemma 3.28]{Yam25})
			\label{lem_chara_UF_linked}
			Let $D$ be an ultrafilter, $\p$ a poset, $Q\subseteq\p$, $\lim^D\colon Q^\omega\to\p$.
			Then, the following are equivalent:
			\begin{enumerate}
				\item $\lim^D$ witnesses $Q$ is $D$-lim-linked.\label{item_lim_1}
				\item $\lim^D$ satisfies $(\star)_n$ below for all $n<\omega$:\label{item_lim_2}
				\begin{align*}
					(\star)_n:& \text{``Given }\bar{q}^j=\langle q_m^j:m<\omega\rangle\in Q^\omega\text{ for }j<n\text{ and }r\leq\textstyle{\lim^D}\bar{q}^j\text{ for all }j<n,\\
					&\text{then }\{m<\omega:r \text{ and all }q_m^j \text{ for } j<n \text{ have a common extension}\}\in D\text{''}.
				\end{align*}
			\end{enumerate}
		\end{lem}

		\begin{dfn} (\cite{Bre95},\cite{BS_E_and_P_2})
			\label{dfn_PR}
			Fix $g\in\left(\omega+1\setminus2\right)^\omega$.
			The $g$-prediction forcing $\prp_g$ consists of tuples $(d,\pi,F)$ satisfying:
			
			\begin{enumerate}
				\item $d\in\sq$.
				\item $\pi=\langle\pi_n:n\in d^{-1}(\{1\})\rangle$.
				\item for each $n\in d^{-1}(\{1\})$, $\pi_n$ is a finite partial function of $\prod_{k<n}g(k)\to g(n)$.
				\item $F\in[\prod_{n<\omega}g(n)]^{<\omega}$
				\item for each $f,f^\prime\in F, f\on|d|=f^\prime\on|d|$ implies $f=f^\prime$.
			\end{enumerate}
			$(d^\prime,\pi^\prime,F^\prime)\leq(d,\pi,F)$ if:
			\begin{enumerate}[(i)]

				\item $d^\prime\supseteq d$.
				\item $\forall n\in d^{-1}(\{1\}), \pi_n^\prime\supseteq\pi_n$.
				\item $F^\prime\supseteq F$.
				\item \label{item_PR_order_long}
				For all $n\in(d^\prime)^{-1}(\{1\})\setminus d^{-1}(\{1\})$ and $f\in F$, we have $f\on n\in\dom(\pi^\prime_n)$ and $\pi^\prime_n(f\on n)=f(n)$.
				
			\end{enumerate}
			When $g(n)=\omega$ for all $n<\omega$, we write $\prp$ instead of $\prp_g$ and just call it ``prediction forcing''.
		\end{dfn}
		
		\begin{thm}(\cite[Corollary 3.34]{Yam25}, \cite{BS_E_and_P_2})
			\label{thm_PR_is_sigma_centered_uf}
			For any $g\in\left(\omega+1\setminus2\right)^\omega$, $\prp_g$ is $\sigma$-$\left( \Lambda(\text{centered})\cap\Lambda^\mathrm{lim}_\mathrm{uf}\right) $-linked.
			Moreover, 
			if $g\in(\omega\setminus2)^\omega$, then 
			$\prp_g$ is $\sigma$-$\left( \Lambda(\text{centered})\cap\Lambda^\mathrm{lim}_\mathrm{cuf}\right)$-linked.
		\end{thm}

		\section{UFI-limits}\label{sec_UFI}
		The following first three subsections are based on \cite{CMU24}.

		\subsection{General theory of UFI-limits}
		
		In this subsection, we present the theory of \textit{ultrafilter-limits for intervals} (UFI-limit), which was first proposed in \cite[Section 4]{Mej24} and later formulated in a more general framework in \cite{CMU24}. While UFI-limit is interpreted as a particular case of their framework (see \cite[Remark 5.43, 6.17, 7.13]{CMU24}), in this paper we give another simpler formulation sufficient for proving the main result. First, unlike the original framework, throughout this paper we fix the interval partition $\bari$ of $\omega$ such that $|I_k|=2^k$ for all $k<\omega$ for simplicity. 
		
		\begin{dfn}(\cite[Definition 4.1]{Mej24})
			\label{dfn_ufi_linked}
			Let $\p$ be a poset.
			\begin{enumerate}
				\item Let $\bar p = \langle p_l\rangle_{l<\omega} \in \p^\omega$ and $k<\omega$. 
				\begin{enumerate}
					\item Define  $\displaystyle\forces_\p\inG_k(\bar{p})\coloneq\{l \in I_k:p_l \in \dot G\}$.
					\item Define $\displaystyle\forces_\p\dns_k(\bar p)\coloneq \frac{|\inG_k(\bar{p})|}{|I_k|}$, called the density of $\bar{p}$ in $I_k$.
				\end{enumerate}

				\item Let 
				$\varepsilon\in(0,1)_\mathbb{Q}\coloneq\{r:r \text{ is a rational number with } 0<r<1\}$ and $D$ be an ultrafilter.
				A set $Q\subseteq\p$ is $(D,\varepsilon)$-lim-linked if there are a $\p$-name $\dot{D}^\prime$ of an ultrafilter extending $D$ and a function $\lim^{D,\varepsilon}\colon Q^\omega\to \p$ such that for any $\bar q = \langle q_l\rangle_{l<\omega} \in Q^\omega$,
				\begin{equation}
					\textstyle{\lim^{D,\varepsilon}}\displaystyle \bar q \Vdash \left\{k<\omega:\dns_k(\bar q)\geq1-\varepsilon\right\} \in \dot {D}^\prime.
				\end{equation}
				When fixing $\varepsilon$, we just write $\lim^D$ instead of $\lim^{D,\varepsilon}$.
				\item Let $\mu$ be an infinite cardinal.
				The poset $\p$ is $\mu$-$\ufi$-linked (short for ultrafilter-limit for intervals) witnessed by $\langle Q_{\alpha,\varepsilon}:\alpha<\mu,\ \varepsilon\in(0,1)_\q\rangle$, if:
				\begin{enumerate}
					\item Each $Q_{\alpha,\varepsilon}$ is $(D,\varepsilon)$-lim-linked for any ultrafilter $D$.
					\item For any $\varepsilon\in(0,1)_\q$, $\bigcup_{\alpha<\mu} Q_{\alpha,\varepsilon}$ is dense in $\p$.
				\end{enumerate}
				Moreover, if each $Q_{\alpha,\varepsilon}$ is centered, we say $\p$ is $\mu$-$\ufic$-lim-linked.

			\end{enumerate}
			When $\mu=\aleph_0$, we use $\mu=\sigma$ instead and often say ``$\p$ has $\ufi$-limits ($\ufic$-limits)'' instead of ``$\p$ is $\sigma$-$\ufi$-lim-linked ($\sigma$-$\ufic$-lim-linked)''.
		\end{dfn}
		
		\begin{rem}
			We introduce the notation $\ufi$ and $\ufic$ to clarify where we use centeredness. For example, on the one hand, we use centeredness in the iteration construction in Lemma \ref{lem_ufi_const_lim}, so we state the lemma using $\ufic$ instead of $\ufi$. On the other hand, we do not require centeredness in the argument itself of controlling $\none$ in Theorem \ref{thm_UFI_keeps_nonE_small}, so we do not use $\ufic$ there. As for centeredness, see also Remark \ref{rem_centered_simplify}, \ref{rem_key_centered}.
			

			

		\end{rem}

		\begin{lem}(\cite[Example 4.3(1)]{Mej24})
			\label{lem_size_UFIC_linked}
			Every singleton $\{q\}\subseteq \p$ is centered and $(D,\varepsilon)$-lim-linked for any ultrafilter $D$ and $\varepsilon\in(0,1)_\mathbb{Q}$. Thus, every poset $\p$ is $|\p|$-$\ufic$-lim-linked and
			in particular, Cohen forcing $\mathbb{C}$ is $\sigma$-$\ufic$-lim-linked.
		\end{lem}

		We will define an iteration for $\ufi$-limits like Definition \ref{dfn_Gamma_iteration}, taking special care of the new parameter $\varepsilon\in(0,1)_\mathbb{Q}$.
		Let $\Gamma_{\mathrm{ufi}}$ represent $\ufi$ or $\ufic$.
		\begin{dfn}
			\label{dfn_UFI_iteration_settings}
			\begin{enumerate}
				\item A $\kappa$-$\Gamma_{\mathrm{ufi}}$-iteration is a fsi $\p_\gamma=\langle(\p_\eta,\qd_\xi):\eta\leq\gamma,\xi<\gamma\rangle $ of ccc forcings with witnesses $\langle\p^-_\xi:\xi<\gamma\rangle$, $\langle\theta_\xi:\xi<\gamma\rangle$ and $\bar{Q}=\langle\dot{Q}^\xi_{\zeta,\varepsilon} : \xi<\gamma,\zeta<\theta_\xi, \varepsilon\in(0,1)_\mathbb{Q}\rangle$ satisfying for all $\xi<\gamma$:
				\begin{enumerate}
					\item $\p^-_\xi\lessdot\p_\xi$.
					\item $\theta_\xi$ is a cardinal of size $<\kappa$.
					\item $\qd_\xi$ and $\langle\dot{Q}^\xi_{\zeta,\varepsilon} : \zeta<\theta_\xi, \varepsilon\in(0,1)_\mathbb{Q}\rangle$ are $\p^-_\xi$-names and $\p^-_\xi$ forces that $\qd_\xi $ is $\theta_\xi$-$\Gamma_{\mathrm{ufi}}$-linked witnessed by
					$\bar{Q}^\xi\coloneq\langle\dot{Q}^\xi_{\zeta,\varepsilon} : \zeta<\theta_\xi, \varepsilon\in(0,1)_\mathbb{Q}\rangle$.
				\end{enumerate}	
				\item $\xi<\gamma$ is a trivial stage if $\Vdash_{\p^-_\xi}|\dot{Q}^\xi_{\zeta,\varepsilon}|=1$ for all $\zeta<\theta_\xi$ and $\varepsilon\in(0,1)_\mathbb{Q}$. $S^-$ is the set of all trivial stages and $S^+\coloneq\gamma\setminus S^-$.
				\item A guardrail is a function $h\in \prod_{\xi<\gamma}(\theta_\xi\times(0,1)_\mathbb{Q})$.
				For a guardrail $h$, we write $h(\xi)=(h_L(\xi),h_R(\xi))$.\label{item_guardrail}
				\item $\p^h_\eta$ is the set of conditions $p\in\p_\eta$ following $h$, i.e., for $\xi\in\dom(p)$,
				$p(\xi)$ is a $\p^-_\xi$-name and $\Vdash_{\p^-_\xi} p(\xi)\in \dot{Q}^\xi_{h(\xi)}$.
				\item $\bar{q}=\langle \dot{q}_l:l<\omega\rangle$  sequentially follows $h$ with constant $\varepsilon\in(0,1)_\mathbb{Q}$ if there are $\langle\xi_l:l<\omega\rangle\in\gamma^\omega$ such that for all $l<\omega$, $h_R(\xi_l)=\varepsilon$ and 
				$\Vdash_{\p^-_{\xi_l}}\dot{q}_l\in \dot{Q}^{\xi_l}_{h(\xi_l)}$.

				\item A set of guardrails $H$ is complete if any countable partial function in $\prod_{\xi<\gamma}(\theta_\xi\times(0,1)_\mathbb{Q})$ is extended to some (total) function in $H$ .
				
			\end{enumerate}
		\end{dfn}
		
		Note that since $|(0,1)_\q|=\aleph_0$, Lemma \ref{lem_complete} also holds for guardrails of this $\Gamma_{\mathrm{ufi}}$-iteration.

		\begin{dfn}
			\label{dfn_UFI_iteration_with_UFs}
			A $\kappa$-$\Gamma_{\mathrm{ufi}}$-iteration has $\Gamma_{\mathrm{ufi}}$-limits on $H$ if
			\begin{enumerate}
				\item $H\subseteq\prod_{\xi<\gamma}(\theta_\xi\times(0,1)_\mathbb{Q})$ is a set of guardrails.
				\item For $h\in H$, $\langle \dot{D}^h_\xi:\xi\leq\gamma \rangle$ is a sequence such that $\dot{D}^h_\xi$ is a $\p_\xi$-name of a non-principal ultrafilter. 
				\item If $\xi<\eta\leq\gamma$, then $\Vdash_{\p_\eta}\dot{D}^h_\xi\subseteq\dot{D}^h_\eta$.
				
				\item For $ \xi\in S^+$, $\Vdash_{\p_\xi} (\dot{D}^h_\xi)^-\in V^{\p^-_\xi}$ where $(\dot{D}^h_\xi)^-\coloneq\dot{D}^h_\xi\cap V^{\p^-_\xi}$.
				\item Whenever $\bar{q}=\langle \dot{q}_l:l<\omega\rangle$ sequentially follows $h$ with constant $\varepsilon\in(0,1)_\mathbb{Q}$ with witnesses $\langle\xi_l:l<\omega\rangle\in\gamma^\omega$, 
				\begin{enumerate}
					\item If $\langle \xi_l:l<\omega \rangle$ is constant with value $\xi$, then
					\begin{equation}
						\label{eq_constant_ufi}
						\Vdash_{\p_\xi}\textstyle{\lim^{(\dot{D}^h_\xi)^-}}\displaystyle\bar{q}\Vdash_{\qd_\xi} \left\{k<\omega:\dns_k(\bar q)\geq 1-\varepsilon\right\} \in\dot{D}^h_{\xi+1}.
					\end{equation}
					(Abusing notation, for $\xi\in S^-$ we use $\lim^{(\dot{D}^h_\xi)^-}\bar{q}$ to denote the constant value of $\bar{q}$ as in the case of UF-limits.)
					\item If $\langle \xi_l:l<\omega \rangle$ is strictly increasing, then
					\begin{equation}
						\label{eq_increasing_ufi}
						\Vdash_{\p_\gamma}\left\{k<\omega:\dns_k(\bar q)\geq 1-\varepsilon\right\} \in\dot{D}^h_\gamma.
					\end{equation}
				\end{enumerate}
				
			\end{enumerate}
		\end{dfn}
		
		We show we can extend ultrafilters at successor steps by using the following fact:
		\begin{fac}(\cite[Lemma 3.20]{BCM21}, \cite[Lemma 3.10]{Yam25})
			\label{lem_UF_upward_directed}
			Let 
			$M\subseteq N$ be transitive models of set theory, $\p\in M$ be a poset, $D_0\in M,D_0^\prime\in N$ be ultrafilters and $\dot{D}_1\in M^\p$ be a name of an ultrafilter.
			If $D_0\subseteq D_0^\prime$ and $\Vdash_{M,\p}D_0\subseteq \dot{D}_1$,
			then there exists $\dot{D}_1^\prime\in N^\p$, a name of an ultrafilter, such that $\Vdash_{N,\p}D_0^\prime,\dot{D}_1\subseteq\dot{D}_1^\prime$.
		\end{fac}

		\begin{lem}(\cite[Theorem 7.15]{CMU24})
			\label{lem_ufi_const_succ}
			Let $\p_{\gamma+1}$ be a $\kappa$-$\Gamma_{\mathrm{ufi}}$-iteration (of length $\gamma+1$) and suppose $\p_\gamma=\p_{\gamma+1}\on\gamma$ has $\Gamma_{\mathrm{ufi}}$-limits on $H$.
			If:
			\begin{equation}
				\label{eq_minus}
				\Vdash_{\p_\gamma} (\dot{D}^h_\gamma)^-\in V^{\p^-_\gamma}\text{ for all }h\in H,
			\end{equation}
			then we can find $\{\dot{D}^h_{\gamma+1}:h\in H\}$ witnessing that $\p_{\gamma+1}$ has $\Gamma_{\mathrm{ufi}}$-limits on $H$.
		\end{lem}
		
		\begin{proof}
			If $\gamma\in S^-$, 
			any $\dot{D}^h_{\gamma+1}$ extending $\dot{D}^h_\gamma$ for $h\in H$ satisfies \eqref{eq_constant_ufi}, so we may assume $\gamma\in S^+$.
			By Definition \ref{dfn_ufi_linked}, for each $h\in H$ we can find a $\p^-_\gamma*\qd_\gamma$-name $\dot{D}^\prime$ of an ultrafilter extending $(\dot{D}^h_\gamma)^-$ such that 
			for any $\bar{q}=\langle \dot{q}_l:l<\omega\rangle$ satisfying $\Vdash_{\p^-_\gamma}\dot{q}_l\in\dot{Q}^\gamma_{h(\gamma)}$ for all $l<\omega$:
			\begin{equation}
				\label{eq_ufconst_succ}
				\Vdash_{\p^-_\gamma}\textstyle{\lim^{(\dot{D}^h_\gamma)^-}}\displaystyle\bar{q}\Vdash_{\dot{Q}_\gamma}\left\{k<\omega:\dns_k(\bar q)\geq 1-\varepsilon\right\} 
				\in\dot{D}^\prime.
			\end{equation}
			Since $(\dot{D}^h_\gamma)^-$ is extended to $\dot{D}^h_\gamma$ and $\dot{D}^\prime$, we can find a  $\p_\gamma*\qd_\gamma=\p_{\gamma+1}$-name $\dot{D}^h_{\gamma+1}$ of an ultrafilter extending $\dot{D}^h_\gamma$ and $\dot{D}^\prime$ by Lemma \ref{lem_UF_upward_directed}.
			This $\dot{D}^h_{\gamma+1}$ satisfies \eqref{eq_constant_ufi} and we are done.
		\end{proof}
		
		The limit step of the construction of ultrafilters is realized particularly when $\Gamma_{\mathrm{ufi}}=\ufic$.
		
		\begin{lem}(\cite[Main Lemma 7.18, Theorem 7.19]{CMU24})
			\label{lem_ufi_const_lim}
			Let $\gamma$ be limit and 
			$\p_\gamma$ be a $\kappa$-$\ufic$-iteration.
			If $\langle \dot{D}^h_\xi:\xi<\gamma, h\in H\rangle$ witnesses that for any $\xi<\gamma$, $\p_\xi=\p_\gamma\on\xi$ has $\ufic$-limits on $H$,  
			then we can find $\langle\dot{D}^h_\gamma:h\in H\rangle$ such that $\langle \dot{D}^h_\xi:\xi\leq\gamma, h\in H\rangle$ witnesses $\p_\gamma$ has $\ufic$-limits on $H$.
		\end{lem}
		
		
		\begin{rem}
			\label{rem_centered_simplify}
			This lemma is obtained from \cite[Main Lemma 7.18]{CMU24} by just interpreting it in our framework, but we give a proof of this lemma for the sake of completeness. Also, as seen in \cite[Theorem 7.19]{CMU24}, centeredness is actually not needed here since $|I_k|\to\infty$ now, but assuming centeredness simplifies the proof and for our purposes the assumption is not too strong in the sense that the new forcing notion $\lebb$ introduced later does satisfy $\sigma$-centerdness (Lemma \ref{lem_LE_sigma_centered}). 
		\end{rem} 
		
		\begin{proof}[Proof of Lemma \ref{lem_ufi_const_lim}]
			For $\cf(\gamma)>\omega$ there is nothing to do, so we assume $\cf(\gamma)=\omega$.
			Let $h\in H$ be arbitrary and $S$ be the collection of $\bar{q}=\langle \dot{q}_l:l<\omega\rangle$ such that for some increasing $\langle \xi_l<\gamma:l<\omega \rangle$ converging to $\gamma$, 
			$\Vdash_{\p^-_{\xi_l}}\dot{q}_l\in\dot{Q}^{\xi_l}_{h(\xi_l)}$ holds for each $l<\omega$.
			For $\bar{q}=\langle \dot{q}_l:l<\omega\rangle\in S$, let $\dot{A}(\bar{q})\coloneq\left\{k<\omega:\dns_k(\bar q)\geq 1-\varepsilon\right\}$.
			We will show:
			\begin{equation}
				\Vdash_{\p_\gamma}\text{``}\bigcup_{\xi<\gamma}\dot{D}^h_\xi\cup\{\dot{A}(\bar{q}):\bar{q}\in S\}\text{ has SFIP}\text{''},
			\end{equation}
			where SFIP is short for Strong Finite Intersection Property and means ``every finite subset has infinite intersection''. 
			To prove this, assume we are given $p\in\p_\gamma$, $k^*<\omega$, $\xi<\gamma$, a $\p_\xi$-name $\dot{A}$ of an element of $\dot{D}^h_\xi$, $\{\bar{q}^i=\langle \dot{q}^i_l:l<\omega\rangle:i<n\}\in[S]^{<\omega}$ and increasing ordinals $\langle \xi_l^i<\gamma:l<\omega \rangle$ converging to $\gamma$ for $i<n$, such that $\Vdash_{\p^-_{\xi^i_l}}\dot{q}_l^i\in\dot{Q}^{\xi^i_l}_{h(\xi^i_l)}$ holds for $l<\omega$ and $i<n$. We will find $q\leq p$ and $k>k^*$ such that:  
			\begin{equation*}
				q\Vdash_{\p_\gamma}k\in\dot{A}\cap \bigcap_{i<n} \dot{A}(\bar{q}^i).
			\end{equation*}
			We may assume that $p\in\p_\xi$ by increasing $\xi$ if necessary.
			Since all $\langle \xi^i_l<\gamma:l<\omega \rangle$ are increasing and converge to $\gamma$, there is $k_0>k^*$ such that $\xi_l^i>\xi$ for any $i<n$, $k>k_0$ and $l\in I_k$.
			By Induction Hypothesis, $p\Vdash_{\p_\xi}\text{``}\dot{D}^h_\xi$ is  an ultrafilter'' and hence we can pick $q\leq_{\p_\xi} p$ and $k>k_0$ 
			such that $q\Vdash_{\p_\xi}k\in\dot{A}$.
			Let us reorder $\{\xi_l^i:i<n,l\in I_k\}=\{\xi^0<\cdots<\xi^{N-1}\}$.
			Inducting on $j<N$, we construct $q_j\in\p_{\xi^j+1}$. 
			Let $q_{-1}\coloneq q$. Assume $j<N$ and we have constructed $q_{j-1}$.
			Let $J_j\coloneq\{(i,l)\in n\times I_k:\xi_l^i=\xi^j\}$.
			Since $\p_{\xi^j}$ forces that all $\dot{q}^i_l$ for $(i,l)\in J_j$ are in the same centered component $\dot{Q}^{\xi^j}_{h(\xi^j)}$, we can pick $p_j\leq q_{j-1}$ in $\p_{\xi^j}$ and 
			a $\p_{\xi^j}$-name $\dot{q}_j$ of a condition in $\qd_{\xi^j}$ such that for each $(i,l)\in J_j$, 
			$p_j\Vdash_{\p_{\xi^j}}\dot{q}_j\leq\dot{q}^i_l$.
			Let $q_j\coloneq p_j^\frown\dot{q}_j$.
			By construction, $q^\prime\coloneq q_{N-1}$ satisfies $q^\prime\leq q\leq p$ and for all $i<n$ and $l\in I_k$, $q^\prime\on\xi_l^i\Vdash_{\p_{\xi_l^i}} q^\prime(\xi_l^i)\leq \dot{q}^i_l$. Thus, for any $i<n$, $q^\prime$ forces $\dot{q}^i_l\in\dot{G}_\gamma$ for any $l\in I_k$ and hence also forces $\dns_k(\bar q^i)=1$. Particularly, $q^\prime\Vdash_{\p_\gamma}k\in \dot{A}(\bar{q}^i)$.
			Hence we obtain $q^\prime\Vdash_{\p_\gamma}k\in\dot{A}\cap \bigcap_{i<n} \dot{A}(\bar{q}^i)$.
		\end{proof}

		\begin{rem}
			\label{rem_key_centered}

			The essential part where we need centeredness is when we afterwards perform an iteration (see Construction \ref{con_P7}) by mixing three kinds of limit methods, UF-limits, cUF-limits and $\ufi$-limits. When extending names of ultrafilters for UF-limits and cUF-limits (not necessarily for $\ufi$-limits as seen in the proof of \cite[Main Lemma 7.17]{CMU24}) at limit steps, we inevitably have to use centeredness, at least within the range of the current known forcing techniques.
		\end{rem}
		
		


		\subsection{Uniform-$(h,\bar{\varepsilon})$-$\Delta$-system.}
		
		We will define $\ufi$-limits for ``refined'' sequences of an iteration. 
		
		\begin{dfn}
			\label{dfn_descending_enum}
			Let $h$ be a guardrail and $\bar{\varepsilon}\colon\omega\to(0,1)_\mathbb{Q}$.
			We say $p\in\p_\gamma$ follows $(h,\bar{\varepsilon})$ if $p$ follows $h$ and $h_R(\a_n)=\bar{\varepsilon}(n)$ for all $n<|\dom(p)|$, where $\{\a_n:n<|\dom(p)|\}=\dom(p)$ is the \textit{descending} enumeration of $\dom(p)$ and the definition of $h_R$ is in Definition \ref{dfn_UFI_iteration_settings} \eqref{item_guardrail}. 
		\end{dfn}
		
		\begin{lem}
			Let $\bar{\varepsilon}\colon\omega\to(0,1)_\mathbb{Q}$ and $H$ be a complete set of guardrails.
			Then, there are densely many conditions following $(h,\bar{\varepsilon})$ for some $h\in H$. 
		\end{lem}
		
		\begin{proof}
			Induct on $\gamma$ recalling that for each $\xi<\gamma$, $\bigcup_{\zeta<\theta_\xi}\dot{Q}^\xi_{\zeta,\varepsilon} \subseteq\qd_\xi$ are forced to be dense for any $\varepsilon\in(0,1)_\mathbb{Q}$. Also note that in the induction proof, the descending enumeration works well. The proof is almost the same as \cite[Lemma 7.5]{CMU24}, so we omit the details.
		\end{proof}
		
		\begin{dfn}(\cite[Definition 8.1]{CMU24})
			\label{dfn_delta_system_UFI}
			Let $\delta$ be an ordinal, 
			$h\in H$, $\bar{\varepsilon}\colon\omega\to(0,1)_\mathbb{Q}$ and 
			$\bar{p}=\langle p_l: l<\delta\rangle\in\p_\gamma^\delta$.
			$\bar{p}$ is an $(h,\bar{\varepsilon})$-uniform $\Delta$-system if:
			\begin{enumerate}
				\item All $p_l$ follow $(h,\bar{\varepsilon})$.
				\item $\{\dom(p_l): l<\delta\}$ forms a $\Delta$-system with some root $\nabla$.
				\item All $|\dom(p_l)|$ are the same $n^\prime$ and $\dom(p_ l)=\{\xi_{n,l}:n<n^\prime\}$ is the \textit{increasing} enumeration.
				\item There is $r^\prime\subseteq n^\prime$ such that $n\in r^\prime\Leftrightarrow\xi_{n,l}\in\nabla$ for $n<n^\prime$.
				\item For $n\in n^\prime\setminus r^\prime$, $\langle\xi_{n, l}: l<\delta \rangle$ is (strictly) increasing.
				
			\end{enumerate}
			Note that for each $n<n^\prime$, $\langle p_l(\xi_{n,l}):l<\delta\rangle$ sequentially follows $h$ with constant $\bar{\varepsilon}(n^\prime-n-1)$, since for each $l<\delta$, $\{\xi_{n,l}:n<n^\prime\}=\dom(p_l)$ is increasing while the enumeration $\{\a_{n,l}:n<n^\prime\}=\dom(p_l)$ in Definition \ref{dfn_descending_enum} is descending.
		\end{dfn}

		The $\Delta$-System Lemma for this uniform $\Delta$-system also holds:
		\begin{lem}(\cite[Theorem 8.3]{CMU24})
			\label{lem_uniform_Delta_System_Lemma}
			Let $h$ be a guardrail, $\bar{\varepsilon}\colon\omega\to(0,1)_\mathbb{Q}$, $\theta$ be an uncountable regular cardinal and $\{p_m:m<\theta\}\subseteq\p_\gamma$ be conditions following $(h,\bar{\varepsilon})$.
			Then, there exists $I\in[\theta]^\theta$ such that $\{p_m:m\in I\}$ forms an $(h,\bar{\varepsilon})$-uniform $\Delta$-system.
		\end{lem}
		
		\begin{dfn}(\cite[Definition 8.5]{CMU24})
			\label{dfn_UFI-limit_condition_for_iteration}
			Let $h\in H$, $\bar{\varepsilon}\colon\omega\to(0,1)_\mathbb{Q}$, $\bar{p}=\langle p_l:l<\omega\rangle\in(\p^h_\gamma)^\omega$ be an $(h,\bar{\varepsilon})$-uniform (countable) $\Delta$-system with root $\nabla$.
			We define the limit condition $p^\infty=\lim^h\bar{p}\in\p_\gamma$ as follows:
			\begin{enumerate}
				\item $\dom(p^\infty)\coloneq\nabla$.
				\item For $\xi\in\nabla$,
				$\Vdash_{\p^-_\xi}p^\infty(\xi):=\lim^{(\dot{D}^h_\xi)^-}\langle p_l(\xi):l<\omega\rangle$.
			\end{enumerate}
		\end{dfn}
		
		\begin{lem}(\cite[Theorem 8.6]{CMU24})
			\label{lem_UFI_principle}
			$p^\infty\Vdash_{\p_\gamma}\{k<\omega:\dns_k(\bar{p})\geq1-\sum_{n<n^\prime}\bar{\varepsilon}(n)\}\in \dot{D}_\gamma^h$.
			
			In particular, for any $\varepsilon\geq\sum_{n<n^\prime}\bar{\varepsilon}(n)$, 
			\begin{equation}
				\label{eq_UFI_principle}
				p^\infty\Vdash_{\p_\gamma}\exists^\infty k<\omega~\dns_{k}(\bar{p})\geq1-\varepsilon.
			\end{equation}
		\end{lem}
		
		For the proof, we use the following elementary combinatorial lemma:
		\begin{lem}
			\label{lem_elem_comb}
			Let $I$ be a non-empty finite set, $n<\omega$, $A_i\subseteq I$ and $\varepsilon_i\in(0,1)_\q$ for each $i<n$.
			Then,
			\begin{equation}
				\label{eq_counting_measure}
				\text{if }\frac{|A_i|}{|I|}\geq 1-\varepsilon_i\text{ for each }i<n, \text{ then }\frac{|\bigcap_{i<n}A_i|}{|I|}\geq1-\textstyle\sum_{i<n}\varepsilon_i.
			\end{equation}
		\end{lem}
		
		\begin{proof}
			Put $B_i\coloneq I\setminus A_i$. Then, \eqref{eq_counting_measure} is equivalent to:
			\begin{equation*}
				\text{if }\frac{|B_i|}{|I|}\leq \varepsilon_i\text{ for each }i<n, \text{ then }\frac{|\bigcup_{i<n}B_i|}{|I|}\leq\textstyle\sum_{i<n}\varepsilon_i,
			\end{equation*}
			which is obvious by the subadditivity of the counting measure.
		\end{proof}


		\begin{proof}[Proof of Lemma \ref{lem_UFI_principle}]
			Induct on $\gamma$.
			
			$\textit{Successor step.}$ Let $\overline{p*q}=\langle(p_l,\dot{q}_l):l<\omega\rangle\in(\p_{\gamma+1})^\omega$ be an $(h,\bar{\varepsilon})$-uniform $\Delta$-system and for this sequence $\overline{p*q}$ we use the same parameters as in Definition \ref{dfn_delta_system_UFI}.	
			To avoid triviality, we may assume that $\gamma\in\nabla$.
			Let $p^\infty\coloneq\lim^h\langle p_l:l<\omega\rangle\in\p_\gamma$ and 
			$\Vdash_{\p^-_\gamma}\dot{q}^\infty\coloneq\lim^{(\dot{D}^h_\gamma)^-}\langle \dot{q}_l:l<\omega\rangle$.
			Let $\varepsilon^+\coloneq\bar{\varepsilon}(0)$ and $\varepsilon^-\coloneq\sum_{0<n<n^\prime}\bar{\varepsilon}(n)$ (note that $n^\prime>0$ since $\gamma\in\nabla$).
			By Induction Hypothesis, $p^\infty\Vdash_{\p_\gamma}\dot{A}\coloneq\{k<\omega:\dns_k(\bar{p})\geq 1-\varepsilon^-\}\in\dot{D}^h_\gamma$.
			By \eqref{eq_constant_ufi}, $\Vdash_{\p_\gamma}\dot{q}^\infty\Vdash_{\qd_\gamma}\dot{B}\coloneq\{k<\omega:\dns_k(\bar{q})\geq 1-\varepsilon^+\}\in\dot{D}^h_{\gamma+1}$.	
			For any $k<\omega$, $\forces_{\p_{\gamma+1}}\inG_k(\bar{p})\cap\inG_k(\bar{q})\subseteq\inG_k(\overline{p*q})$ holds, so
			$p^\infty*\dot{q}^\infty$ forces that for all $k\in\dot{A}\cap\dot{B}$, $\displaystyle\frac{|\inG_k(\overline{p*q})|}{|I_k|}\geq\frac{|\inG_k(\bar{p})\cap\inG_k(\bar{q})|}{|I_k|}\geq1-(\varepsilon^-+\varepsilon^+)=1-\textstyle\sum_{n<n^\prime}\bar{\varepsilon}(n)$ by Lemma \ref{lem_elem_comb} and hence we obtain 
			$p^\infty*\dot{q}^\infty\forces_{\p_{\gamma+1}}\{k<\omega:\dns_k(\overline{p*q})\geq 1-\sum_{n<n^\prime}\bar{\varepsilon}(n)\}\supseteq\dot{A}\cap\dot{B}\in\dot{D}^h_{\gamma+1}$.

			%
			%
			%
			%
			%
			
			$\textit{Limit step.}$ Let $\gamma$ be limit and $\bar{p}=\langle p_l:l<\omega\rangle\in(\p^h_\gamma)^\omega$ be an $(h,\bar{\varepsilon})$-uniform-$\Delta$-system and for this sequence $\bar{p}$ we use the same parameters as in Definition \ref{dfn_delta_system_UFI}.
			Let $\xi\coloneq\sup(\nabla)+1<\gamma$ (note that $\nabla$ might be empty) and $\bar{p}\on\xi\coloneq\langle p_l\on\xi:l<\omega\rangle$.
			Since 
			$\bar{p}\on\xi$ is also an $(h,\bar{\varepsilon})$-uniform $\Delta$-system with root $\nabla$,
			$p^\infty\coloneq\lim^h\bar{p}\on\xi=\lim^h\bar{p}$ by definition.
			Let $n^{\prime\prime}\coloneq\sup({r^\prime})+1$.
			By Induction Hypothesis, $p^\infty\Vdash_{\p_\xi}\dot{A}\coloneq\{k<\omega:\dns_k(\bar{p}\on\xi)\geq 1-\sum_{n<n^{\prime\prime}}\bar{\varepsilon}(n^\prime-n-1)\}\in\dot{D}^h_\xi$.
			Since $\langle \xi_{n,l}:l<\omega \rangle$ is increasing for $n\in\left[n^{\prime\prime},n^\prime\right)$,
			by \eqref{eq_increasing_ufi}, 
			$\Vdash_{\p_\gamma}\dot{B}_n\coloneq\{k<\omega:\dns_k(\bar{p}_n)\geq1-\bar{\varepsilon}(n^\prime-n-1)\}\in\dot{D}^h_\gamma$ for $n\in\left[n^{\prime\prime},n^\prime\right)$ where $\bar{p}_n\coloneq\langle p_l(\xi_{n,l}):l<\omega\rangle$, and hence
			$p^\infty$ forces $\dot{X}\coloneq\dot{A}\cap\bigcap_{n\in\left[n^{\prime\prime},n^\prime\right)} \dot{B}_n\in \dot{D}^h_\gamma$. 
			Since $\dom(p_l)=\dom(p_l\on\xi)\cup\{\xi_{n,l}:n\in\left[n^{\prime\prime},n^\prime\right)\}$ for any $l<\omega$, 
			$\forces_{\p_{\gamma}}\inG_k(\bar{p})\supseteq\inG_k(\bar{p}\on\xi)\cap \bigcap_{n\in\left[n^{\prime\prime},n^\prime\right)}\inG_k(\bar{p}_n)$ for any $k<\omega$. 
			Thus by Lemma \ref{lem_elem_comb}, $p^\infty$ forces that for each $k\in\dot{X}$, $\dns_k(\bar{p})\geq1-\left(\sum_{n<n^{\prime\prime}}\bar{\varepsilon}(n^\prime-n-1)+\sum_{n\in\left[n^{\prime\prime},n^\prime\right)}\bar{\varepsilon}(n^\prime-n-1)\right)=1-\sum_{n<n^{\prime}}\bar{\varepsilon}(n)$ and hence $p^\infty\forces_{\p_\gamma}\{k<\omega:\dns_k(\bar{p})\geq1-\sum_{n<n^\prime}\bar{\varepsilon}(n)\}\supseteq\dot{X}\in \dot{D}^h_\gamma$.
		\end{proof}

		\subsection{Application to $\none$}
		
		To control $\none$ using $\ufi$-limits, we first iterate Cohen forcings as in the case of the standard ultrafilter-limits:
		\begin{ass}
			\label{ass_first_Cohen_UFI}
			\begin{enumerate}
				\item $\kappa\leq \lambda$ are uncountable cardinals, $\kappa$ is regular and $\gamma=\lambda+\lambda$.
				\item $\p_\gamma$ is a $\kappa$-$\Gamma_\mathrm{ufi}$-iteration with $\Gamma_\mathrm{ufi}$-limits on $H$ (with the same parameters as in Definition \ref{dfn_UFI_iteration_settings} and \ref{dfn_UFI_iteration_with_UFs}).
				\item $H$ is complete and $|H|<\kappa$.
				\item For $\xi<\lambda$, $\Vdash_{\p^-_\xi}\qd_\xi=\mathbb{C}$, Cohen forcing. Note that $\mathbb{C}$ is $\kappa$-$\Gamma_\mathrm{ufi}$-linked by Example \ref{lem_size_UFIC_linked}.
			\end{enumerate}
		\end{ass}
		We show that UFI-limits keep $\none$ small:
		\begin{thm}(\cite[Theorem 10.4]{CMU24})
			\label{thm_UFI_keeps_nonE_small}
			Suppose Assumption \ref{ass_first_Cohen_UFI}. 
			Then, $\p=\p_\gamma$ forces $C_{[\lambda]^{<\kappa}}\lq C_\mathcal{E}$, in particular, $\none\leq\kappa$.
		\end{thm}
		
		The theorem is obtained from \cite[Theorem 10.4]{CMU24} by just interpreting it in our framework, but we give a proof of it for the sake of completeness.

		\begin{proof}
			We may assume that $\mathcal{E}$ is defined in $2^\omega$ (see footnote \ref{footnote:E}). 
			We show that the first $\lambda$-many Cohen reals $\{\dot{c}_\b:\b<\lambda\}$ witness $C_{[\lambda]^{<\kappa}}\lq C_\mathcal{E}$.
			Assume towards contradiction that there exist a $\p$-name $\dot{C}$ of a member of $\mathcal{E}$ and a condition $p\in\p$ such that $p\forces |\{\b<\lambda:\dot{c}_\b\in \dot{C}\}|\geq\kappa$. 
			Since $\forces\dot{C}\in\mathcal{E}$, there exist $\p$-names $\{\dot{A}_j:j<\omega\}$ of closed null sets in $2^\omega$ such that $\forces\dot{C}\subseteq\bigcup_{j<\omega}\dot{A}_j$.
			Since each $\dot{A}_j$ is closed in $2^\omega$, there is a $\p$-name $\dot{T}_j$ of a tree in $2^{<\omega}$ such that $\dot{A}_j=[\dot{T}_j]=\{x\in2^\omega:x\on n \in \dot{T}_j \text{ for }n<\omega\}$. For each $\a<\kappa$, inductively take $p_\a\leq p$, $\b_\a<\lambda$ and $j_\a<\omega$ such that 
			$p_\a\forces \b_\a\in\{\b<\lambda:\dot{c}_\b\in \dot{C}\}\setminus\{\b_{\a^\prime}:\a^\prime<\a\}\text{ and }\dot{c}_{\b_\a}\in \dot{A}_{j_\a}=[\dot{T}_{j_\a}]$, which is possible since $p\forces |\{\b<\lambda:\dot{c}_\b\in \dot{C}\}|\geq\kappa$ and $\a<\kappa$.
			
			Fix some $\bar{\varepsilon}\colon\omega\to(0,1)_\mathbb{Q}$ with $\varepsilon\coloneq\sum_{n<\omega}\bar{\varepsilon}(n)<1$. 
			By extending and thinning, 	we may assume the following: 
			\begin{enumerate}
				\item $\b_\a\in\dom(p_\a)$.
				\item $p_\a$ follows some $(h_\a,\bar{\varepsilon})$ where $h_\a\in H$.
				\item All $h_\a$ are the same $h$.
				\item $\langle p_\a:\a<\kappa\rangle$ forms an $(h,\bar{\varepsilon})$-uniform $\Delta$-system with root $\nabla$.
				\item All $\dot{T}_{j_\a}$ are the same $\dot{T}$.
				\item All $p_\a(\b_\a)$ are the same Cohen condition $s\in\sq$.
			\end{enumerate}
			Note that $\b_\a\notin\nabla$ since all $\b_\a$ are distinct. Pick the first $\omega$-many $p_\a$.
			Fix some bijection $\sigma\colon\omega\to\sq$ such that for any $k<\omega$, $\{\sigma(l):l\in I_k\}=2^k$ (recall $|I_k|=2^k$).
			For each $l<\omega$, define $q_l\leq p_l$ by extending the $\b_l$-th position to $q_l(\b_l)\coloneq s^\frown\sigma(l)$.
			Since $\b_l\notin\nabla$, all $q_l$ follow a new common partial countable guardrail, which is extended to some $h^\prime\in H$ by the completeness of $H$.
			Thus, $\bar{q}\coloneq\langle q_l:l<\omega\rangle$ forms an $(h^\prime,\bar{\varepsilon})$-uniform $\Delta$-system with root $\nabla$ (note that $\bar{\varepsilon}$ is the same since $\b_l<\kappa$ are trivial stages and hence Lemma \ref{lem_size_UFIC_linked} can be applied), so we can take the UFI-limit $q^\infty\coloneq\lim^{h^\prime}\bar{q}$.
			By \eqref{eq_UFI_principle}, we have:
			\begin{equation}
				\label{eq_q_ufi}
				q^\infty\Vdash_{\p_\gamma}\exists^\infty k<\omega~\dns_{k}(\bar{q})\geq1-\varepsilon,
			\end{equation}
			since $\varepsilon>\sum_{n<n^{\prime}}\bar{\varepsilon}(n)$ where $n^\prime$ is the same size of all $\dom(q_l)$.
			By the definition of $q_l$, note that for each $k<\omega$ and $l\in I_k$, we have:
			\begin{equation}
				\label{eq_ql}
				q_l\forces\dot{c}_{\b_l}\in[\dot{T}]\text{ and } \dot{c}_{\b_l}\on (|s|+k)=s^\frown\sigma(l), \text{ so }s^\frown\sigma(l)\in\dot{T}.
			\end{equation}
			Let $G$ be any $\p$-generic filter containing $q^\infty$ and work in $V[G]$.
			For $k<\omega$, let $L_k\coloneq\inG_k(\bar{q})[G]=\{l\in I_k:q_l\in G\}$.
			By \eqref{eq_q_ufi} and \eqref{eq_ql}, we have:
			\begin{equation}
				\exists^\infty k<\omega~\dns_{k}(\bar{q})[G]=\frac{|L_k|}{|I_k|}\geq1-\varepsilon\text{ and for }l\in L_k, s^\frown\sigma(l)\in\dot{T}[G].
			\end{equation}
			Thus, for infinitely many $k<\omega$, we obtain (recall $|I_k|=2^k$):
			\begin{equation}
				\label{eq_measure}
				\dfrac{|\dot{T}[G]\cap[s]\cap2^{|s|+k}|}{2^k}\geq 1-\varepsilon,
			\end{equation}
			hence \eqref{eq_measure} holds for all $k<\omega$ since the value on the left-hand side is non-increasing in $k$.
			Therefore, we have $\dfrac{\mu([T])}{\mu([s])}\geq 1-\varepsilon>0$ where $\mu$ denotes the Lebesgue measure on $2^\omega$, which contradicts that $[T]$ is null.
		\end{proof}
		
		\subsection{Poset with $\ufic$-limits and $\sigma$-centeredness}
		In this subsection, we introduce the forcing notion $\lebb$ which satisfies all the required properties mentioned in Section \ref{sec_intro}, and this is the main contribution of this paper. 
		
		First, fix a function $b\in\oo$ with $b(n)\geq \exp(n)\coloneq 2^n$ for all $n<\omega$.
		
		\begin{dfn}
			For $x,y\in\prod b$ and $D\in\ooo$, $x\neq^*(D,y)$ if $x(n)\neq y(n)$ for all but finitely many $n\in D$.
			Define the relational system $\lebf=\lebf_b\coloneq\langle \prod b, \ooo\times\prod b,\neq^*\rangle$ (short for ``locally eventually different'').
		\end{dfn}
		

		\begin{lem}
			$C_\m\lq\lebf$. In particular, $\bb(\lebf)\leq\nonm$.
		\end{lem}
		
		\begin{proof}
			We may consider the meager ideal $\m$ in the product space $\prod b$ of discrete spaces $\{ b(n):n<\omega\}$.
			For any $m<\omega$ and $(D,y)\in\ooo\times\prod b$, the set $\bigcap_{m\leq n\in D}\{x\in\prod b: x(n)\neq y(n)\}$ is closed nowhere dense in the space $\prod b$, which induces the Tukey order $C_\m\lq\lebf$. 
		\end{proof}
		
		\begin{rem}
			$\bb(\lebf)<\nonm$ consistently holds: Brendle and L{\"o}we's work \cite[Theorem 12]{BL11} implies that Hechler forcing is ``$\lebf$-good'', (a property concerning the preservation of $\lebf$-unbounded families in the context of fsi, see e.g. \cite{JS90}, \cite{Bre91}, \cite{CM19}), and hence $\aleph_1=\bb(\lebf)<\bb=\nonm=2^{\aleph_0}$ holds in the Hechler model.
		\end{rem}
		
		
		Thus, in the upcoming forcing iteration construction we use this relational system $\lebf$ instead to calculate the value of $\nonm$. 
		For this purpose, we introduce the following poset $\lebb$ which generically adds an $\lebf$-dominating real and hence increases $\bb(\lebf)\leq\nonm$:

		\begin{dfn}
			\begin{enumerate}
				\item The poset $\lebb=\lebb_b$ is defined as follows:
				\begin{enumerate}
					\item The conditions are triples $(d,s,\varphi)$ where $d\in\sq$, $s\in\prod_{n<|d|} b(n)$ and $\varphi\in\prod_{n<\omega}\mathcal{P}(b(n))$ such that some $k<\omega$ satisfies:
					\begin{equation}
						\dfrac{|\varphi(n)|}{b(n)}\leq\exp\left(-\dfrac{n}{2^k}\right)\text{ for all }n\geq|d|.
					\end{equation} 
					\item $(d^\prime,s^\prime,\varphi^\prime)\leq(d,s,\varphi)$ if $d^\prime\supseteq d$, $s^\prime\supseteq s, \varphi^\prime(n)\supseteq\varphi(n)$ for $n<\omega$ and:
					\begin{equation}
						\text{for all }n\in(d^\prime)^{-1}(\{1\})\setminus d^{-1}(\{1\}), s^\prime(n)\notin\varphi(n).
					\end{equation}
				\end{enumerate}
				
				\item Put $Q_k\coloneq\left\{(d,s,\varphi)\in\lebb:\dfrac{|\varphi(n)|}{b(n)}\leq\exp\left(-\dfrac{n}{2^k}\right)\text{ for }n\geq|d|\right\}$ for $k<\omega$ (thus $\p=\bigcup_{k<\omega}Q_k$).
				Note $Q_k\subseteq Q_{k^\prime}$ for $k<k^\prime$.
				\item In the generic extension $V[G]$, define $D_G\coloneq\bigcup_{(d,s,\varphi)\in G}\din$ and  $y_G\coloneq\bigcup_{(d,s,\varphi)\in G}s$.
			\end{enumerate}

		\end{dfn}

		We introduce some useful notation:
		\begin{dfn}
			
			\begin{itemize}
				\item For $N<\omega$, let $\zero_N $ denote the sequences of length $N$ whose values are all $0$. 
				Namely, $\zero_N \coloneq N\times \{0\}$. 
				\item For $d,d^\prime\in\sq$ with $d^\prime\supseteq d$, $\New(d^\prime,d)\coloneq\newd$. 
			\end{itemize}
			
		\end{dfn}
		\begin{lem}
			\label{lem_LE_extend}
			Let $k<\omega$ and $(d,s,\varphi)\in Q_k$. Then, for any $m<\omega$, there is an extension $(d^\prime,s^\prime,\varphi)\leq(d,s,\varphi)$ such that $d^\prime(n)=1$ for some $n> m$.
			(Note that $(d^\prime,s^\prime,\varphi)\in Q_k$ here).
			
			In particular, $D_G$ is (forced to be) infinite.
		\end{lem}
		\begin{proof}
			Let $d^\prime\coloneq d^\frown(\zero_N)^\frown1$ where $N$ is so large that $n\coloneqq|d^\prime|-1>m$.
			Take some extension $s^\prime\in\prod_{i<|d^\prime|} b(i)$ of $s$ such that $s^\prime(n)\in b(n)\setminus\varphi(n)$, which is possible since $\dfrac{|\varphi(n)|}{b(n)}\leq\exp\left(-\dfrac{n}{2^k}\right)<1$ (note $n>m\geq0$).
			Then, $(d^\prime,s^\prime,\varphi)$ is as required.
		\end{proof}
		

		\begin{lem}
			\label{lem_LE_increases_bbLE}
			$(D_G,y_G)$ is a generic $\lebf$-dominating real and hence $\lebb$ increases $\bb(\lebf)$ (by iteration).
		\end{lem}
		\begin{proof}
			Since $D_G$ is already forced to be infinite, we show that for $x\in\prod b$,  $E\coloneq\{(d,s,\varphi):x(n)\in\varphi(n)\text{ for }n\geq |d|\}$ is dense.
			Given $(d,s,\varphi)\in Q_k$ for some $k$,
			we may assume that $2\leq\exp\left(\dfrac{|d|}{2^{k+1}}\right)$ by Lemma \ref{lem_LE_extend}.
			Define a function $\psi$ on $\omega$ by $\psi(n)\coloneq\varphi(n)\cup\{x(n)\}$.
			For $n\geq|d|$, we have
			$\dfrac{|\psi(n)|}{b(n)}\leq\dfrac{|\varphi(n)|+1}{b(n)}\leq\exp\left(-\dfrac{n}{2^k}\right)+\exp\left(-n\right)\leq 2\exp\left(-\dfrac{n}{2^k}\right)\leq\exp\left(\dfrac{|d|}{2^{k+1}}\right)\exp\left(-\dfrac{n}{2^k}\right)\leq\exp\left(-\dfrac{n}{2^{k+1}}\right)$ since $\exp(n)\leq b(n)$, so $(d,s,\psi)$ is a valid condition and in $E$.
		\end{proof}
		
		The following lemma is easy to prove, while it plays a key role in the upcoming forcing construction as mentioned in Remark \ref{rem_key_centered}.
		\begin{lem}
			\label{lem_LE_sigma_centered}
			For any $d\in\sq$, $s\in\prod_{n<|d|} b(n)$ and $k<\omega$,  $Q=Q_{d,s,k}\coloneq\{(d^\prime,s^\prime,\varphi)\in Q_k:d^\prime=d, s^\prime=s\}$ is centered.
			In particular, $\lebb$ is $\sigma$-centered.
		\end{lem}
		\begin{proof}
			Assume $l<\omega$ and $\langle r_i=(d,s,\varphi_i)\rangle_{i<l}\in Q^l$ are given. Let $d^\prime\coloneq d^\frown\zero_N$ where $N$ is so large that $l\leq\exp\left(\dfrac{|d^\prime|}{2^{k+1}}\right)$ holds and let $s^\prime\in\prod_{n<|d^\prime|} b(n)$ be any extension of $s$. Define a function $\psi$ on $\omega$ by $\psi(n)\coloneq\bigcup_{i<l}\varphi_i(n)$ for $n<\omega$.
			For $n\geq|d^\prime|$, 
			$\dfrac{|\psi(n)|}{b(n)}\leq l\exp\left(-\dfrac{n}{2^k}\right)\leq \exp\left(\dfrac{|d^\prime|}{2^{k+1}}\right)\exp\left(-\dfrac{n}{2^k}\right)\leq\exp\left(-\dfrac{n}{2^{k+1}}\right)$, so $(d^\prime,s^\prime,\psi)$ is a valid condition and extends all $r_i$ for $i<l$, since $\New(d^\prime,d)=\emptyset$.
		\end{proof}
		
		\begin{rem}
			This is the point where the first coordinate $d$ is essentially needed. In fact, without this parameter $d$, the corresponding forcing notion will also satisfy the three other required properties: it increases $\nonm$, it has UF-limits and it has UFI-limits, but will not be $\sigma$-centered. However, once we add this $d$, the $\sigma$-centeredness is satisfied, without any loss of the other properties.
		\end{rem}
		
		$\lebb$ has (the standard) closed-UF-limits:
		\begin{lem}
			\label{lem_LE_cUF}
			Each $Q_{d,s,k}$ is closed-$D$-limit-linked for any ultrafilter $D$.
			In particular, 
			$\lebb$ is $\sigma$-$\left(\Lambda(\mathrm{centered})\cap\Lambda^\mathrm{lim}_\mathrm{cuf}\right)$-lim-linked. 
		\end{lem}
		
		\begin{proof}
			Put $Q\coloneq Q_{d,s,k}$. 
			Given $\bar{q}=\langle q_m=(d,s,\varphi_m)\rangle_{m<\omega}\in Q^\omega$,
			define $q^\infty=\lim^D\bar{q}=(d,s,\varphi_\infty)\in\lebb$ by:
			\begin{equation}
				\label{eq_def_of_uflimit_of_LE}
				A=\varphi_\infty(n):\Leftrightarrow \{m<\omega:A=\varphi_m(n)\}\in D,
			\end{equation}
			for $n<\omega$ and $A\subseteq b(n)$, which is possible since $\mathcal{P}(b(n))$ is finite.
			It is easy to see that $q^\infty\in Q_k$ and hence in $Q$, since each $\varphi_\infty(n)$ is equal to some $\varphi_m(n)$. Thus, the closedness of the limit function is proved.
			
			To show $(\star)_N$ in Lemma \ref{lem_chara_UF_linked} for $N\geq 1$, assume that $\bar{q}^j=\langle q^j_m=(d,s,\varphi_m^j)\rangle_{m<\omega}\in Q^\omega$ for $j<N$ and $r=(d^\prime,s^\prime,\varphi^\prime)\leq\lim^D\bar{q}^j=(d,s,\varphi_\infty^j)$ for all $j<N$.
			Let $k^\prime<\omega$ witness $r\in Q_{k^\prime}$ and we may assume that $k^\prime\geq k$ and $N+1\leq\exp\left(\dfrac{|d^\prime|}{2^{k^\prime+1}}\right)$ by extending.
			For $j<N$ and $n\in\New(d^\prime,d)$, let $X_{j,n}\coloneq\{m<\omega:\varphi^j_m(n)=\varphi_\infty^j(n)\}\in D$ and $X\coloneq\bigcap\{X_{j,n}:j<N, n\in\New(d^\prime,d)\}\in D$. Let $m\in X$ and we show $r$, $\{q^j_m:j<N\}$ have a common extension. 
			Define a function $\psi$ on $\omega$ by $\psi(n)\coloneq\varphi^\prime(n)\cup\bigcup_{i<N}\varphi_m^i(n)$ for $n<\omega$ and let $r^\prime\coloneq(d^\prime,s^\prime,\psi)$.
			For $n\geq|d^\prime|$, 
			$\dfrac{|\psi(n)|}{b(n)}\leq(N+1)\exp\left(-\dfrac{n}{2^{k^\prime}}\right)\leq \exp\left(\dfrac{|d^\prime|}{2^{k^\prime+1}}\right)\exp\left(-\dfrac{n}{2^{k^\prime}}\right)\leq\exp\left(-\dfrac{n}{2^{k^\prime+1}}\right)$, so $r^\prime=(d^\prime,s^\prime,\psi)$ is a valid condition and obviously extends $r=(d^\prime,s^\prime,\varphi^\prime)$. 
			For $j<N$, since $r=(d^\prime,s^\prime,\varphi^\prime)\leq(d,s,\varphi_\infty^j)=\lim^D\bar{q}^j$, $s^\prime(n)\notin \varphi_\infty^j(n)=\varphi_m^j(n)$ for $n\in\New(d^\prime,d)$.
			Thus, $r^\prime$ also extends all $q^j_m$ for $j<N$.
		\end{proof}
		
		To show that $\lebb$ has UFI-limits, we use the notion of ``pseudo-fusion'' introduced in \cite{Mej24}. 
		\begin{dfn}(\cite[Lemma 3.12]{Mej24})
			Let $\p$ be a poset.
			
			\begin{enumerate}
				\item Let $0<L<\omega$, $\bar{p}\coloneq\langle p_l:l<L\rangle\in\p^L$ and $\varepsilon\in(0,1)_\mathbb{Q}$.
				A condition $q\in\p$ is an $\varepsilon$-pseudo-fusion of $\bar{p}$ if:
				\begin{equation*}
					q\Vdash \dfrac{|\{l<L:p_l\in\dot{G}\}|}{L}\geq1-\varepsilon.
				\end{equation*}
				\item  For $\varepsilon\in(0,1)_\mathbb{Q}$, $Q\subseteq\p$ is $\varepsilon$-pseudo-fusion-linked if for any $0<L<\omega$ and $\bar{p}\coloneq\langle p_l:l<L\rangle\in Q^L$, there is an $\varepsilon$-pseudo-fusion of $\bar{p}$.
			\end{enumerate}
			We often say $\p$ has pseudo-fusions if for any $\varepsilon\in(0,1)_\mathbb{Q}$, $\p$ has a dense subposet which is a countable union of $\varepsilon$-pseudo-fusion-linked components. 
		\end{dfn}

		We will show that $\lebb$ has pseudo-fusions:
		\begin{dfn}
			\label{dfn_Rdsk}
			We say $d\in\sq$ is long if $|d|\geq 2 \text{ and }\exp\left(-\dfrac{n}{2^k}\right)\leq\dfrac{1}{n^2(n+1)^2}\text{ holds for }n\geq|d|$.
			Note that for each $\varepsilon\in(0,1)_\mathbb{Q}$, the set:
			\[\bigcup\left\{Q_{d,s,k}:d\in\sq\text{ is long and }\frac{1}{|d|}\leq\varepsilon,s\in\prod_{n<|d|}b(n),k<\omega\right\}\]
			is dense in $\lebb$ since the inequality  $\exp\left(-\dfrac{n}{2^k}\right)\leq\dfrac{1}{n^2(n+1)^2}$ holds for all but finitely many $n$.
%

%
%
%
%
		\end{dfn}
		
		\begin{lem}
			\label{lem_LE_has_pseudo_fusions}
			For each long $d$, $Q_{d,s,k}$ is $\frac{1}{|d|}$-pseudo-fusion-linked and moreover the pseudo-fusion conditions are always in $Q_{d,s,k+1}$.
			In particular, $\lebb$ has pseudo-fusions.
		\end{lem}
		

		To prove this lemma, we use the following combinatorial lemma:
		\begin{lem}
			\label{lem_fusion_principle}
			Suppose:
			\begin{enumerate}
				\item $0<M,L<\omega$, $0<\varepsilon\leq \delta<1$.
				\item For $l<L$, $A_l\subseteq M$ and $\dfrac{|A_l|}{M}\geq1-\varepsilon$.
				\item For $m<M$, $a_m\coloneq\{l<L:m\in A_l\}$ and $\a_m\coloneq\dfrac{|a_m|}{L}$.
				\item $X\coloneq\{m<M:\a_m\geq1-\delta\}$, $x\coloneq\dfrac{|X|}{M}$
			\end{enumerate}
			Then, $x\geq1-\dfrac{\varepsilon}{\delta}$.
		\end{lem}
		\begin{proof}
			If $X=\emptyset$, then $1-\delta\geq\dfrac{\sum_{m<M}\a_m }{M}=\dfrac{\sum_{m<M}|a_m|}{ML}=\dfrac{\sum_{l\in L}|A_l|}{ML}\geq1-\varepsilon$, so $\varepsilon=\delta$ holds and we have $x=0\geq 1-\dfrac{\varepsilon}{\delta}=0$.
			Assume that $X$ and $Y\coloneq M\setminus X$ are not empty.
			Let $\a\coloneq\dfrac{\sum_{m\in X}\a_m}{|X|}$ and $\b\coloneq\dfrac{\sum_{m\in Y}\a_m}{|Y|}$ and note
			\begin{equation}
				\label{eq_a_delta_b}
				1\geq\a\geq1-\delta>\b\geq0.
			\end{equation}
			Since $\a|X|L+\b|Y|L=\sum_{m<M}|a_m|=\sum_{l\in L}|A_l|\geq(1-\varepsilon)ML$,
			we obtain
			$\a x+\b(1-x)\geq1-\varepsilon$, which implies $(\a-\b)x\geq1-\b-\varepsilon$. By \eqref{eq_a_delta_b}, $(1-\b)x\geq(\a-\b)x\geq1-\b-\varepsilon$ and hence we have (again by \eqref{eq_a_delta_b}): \[x\geq1-\dfrac{\varepsilon}{1-\b}\geq1-\dfrac{\varepsilon}{1-(1-\delta)}=1-\dfrac{\varepsilon}{\delta}.\] 
			
%
		\end{proof}

		
		\begin{proof}[Proof of Lemma \ref{lem_LE_has_pseudo_fusions}]
			Given $0<L<\omega$ and $\bar{r}=\langle r_l=(d,s,\varphi_l)\rangle_{l<L}\in (Q_{d,s,k})^L$, 
			we define $r^*=\pf(\bar{r})=(d,s,\varphi^*)$ 
			as follows:
			For each $n\geq|d|$, define $\varphi^*(n)$ by Lemma \ref{lem_fusion_principle} by putting $M\coloneq  b(n), A_l\coloneq b(n)\setminus\varphi_l(n), \varepsilon\coloneq\varepsilon_n\coloneq\displaystyle{\max_{l<L}}\dfrac{|\varphi_l(n)|}{b(n)}$, $\delta\coloneq\sqrt{\varepsilon_n}\geq\varepsilon_n$, $\varphi^*(n)\coloneq b(n)\setminus X$.
			Note that by definition, for $n\geq|d|$ we have:
			\begin{equation}
				\label{eq_varp_star}
				\varphi^*(n)=\left\{m<b(n):\dfrac{|\{l<L:m\notin\varphi_l(n)\}|}{L}<1-\sqrt{\varepsilon_n}\right\},\text{ and}
			\end{equation}
			\begin{equation}
				\label{eq_vare}
				\varepsilon_n\leq\exp\left(-\dfrac{n}{2^k}\right)\leq\dfrac{1}{n^2(n+1)^2}.
			\end{equation}
			First we show that $r^*=(d,s,\varphi^*)$ is a valid condition and in $Q_{k+1}$. By Lemma \ref{lem_fusion_principle} and \eqref{eq_vare}, for $n\geq|d|$, we have  $\dfrac{|\varphi^*(n)|}{b(n)}\leq\dfrac{\varepsilon_n}{\sqrt{\varepsilon_n}}=\sqrt{\varepsilon_n}\leq\sqrt{\exp\left(-\dfrac{n}{2^k}\right)}=\exp\left(-\dfrac{n}{2^{k+1}}\right)$, so $r^*=(d,s,\varphi^*)$ is a condition in $ Q_{k+1}$.
			
			Then, we show $r^*$ is a $\frac{1}{|d|}$-pseudo-fusion of $\bar{r}$, by proving that $D\coloneq\left\{r:\dfrac{|\{l<L:r\leq r_l\}|}{L}\geq1-\dfrac{1}{|d|}\right\}$ is dense below $r^*$. Let $r=(d^\prime,s^\prime,\varphi^\prime)\leq r^*$ be arbitrary and $k^\prime<\omega$ witness $r\in Q_{k^\prime}$. We may assume that $k^\prime\geq k$ and moreover $L+1\leq\exp\left(\dfrac{|d^\prime|}{2^{k^\prime+1}}\right)$ by extending. 
			For $n\in\New(d^\prime,d)$, let $S_n\coloneq\{l<L:s^\prime(n)\notin\varphi_l(n)\}$.
			For such $n$, since $(d^\prime,s^\prime,\varphi^\prime)=r\leq r^*=(d,s,\varphi^*)$, $s^\prime(n)\notin\varphi^*(n)$ holds, which implies $\dfrac{|\{l<L:s^\prime(n)\notin \varphi_l(n)\}|}{L}=\dfrac{|S_n|}{L}\geq1-\sqrt{\varepsilon_n}\geq 1-\dfrac{1}{n(n+1)}$ by \eqref{eq_varp_star} and \eqref{eq_vare}.
			Thus, $S\coloneq\displaystyle\bigcap_{n\in\New(d^\prime,d)} S_n$ satisfies by Lemma \ref{lem_elem_comb}:
			\begin{equation*}
				\dfrac{|S|}{L} \geq\displaystyle1-\sum_{n\in\New(d^\prime,d)}\dfrac{1}{n(n+1)} 
				\geq1-\sum_{n\geq|d|}\dfrac{1}{n(n+1)}
				=1-\dfrac{1}{|d|}.
			\end{equation*}
			Therefore, it suffices to show that some extension $r^\prime \leq r$ also extends all $r_l$ for $l\in S$, so particularly $r^\prime$ is in $D=\left\{r^{\prime\prime}:\dfrac{|\{l<L:r^{\prime\prime}\leq r_l\}|}{L}\geq1-\dfrac{1}{|d|}\right\}$.
			Define a function $\psi$ on $\omega$ by $\psi(n)\coloneq\varphi^\prime(n)\cup\bigcup_{i\in S}\varphi_i(n)$ for $n<\omega$.
			For $n\geq|d^\prime|$, we have 
			$\dfrac{|\psi(n)|}{b(n)}\leq(1+|S|)\exp\left(-\dfrac{n}{2^{k^\prime}}\right)\leq(1+L)\exp\left(-\dfrac{n}{2^{k^\prime}}\right)\leq \exp\left(\dfrac{|d^\prime|}{2^{k^\prime+1}}\right)\exp\left(-\dfrac{n}{2^{k^\prime}}\right)\leq\exp\left(-\dfrac{n}{2^{k^\prime+1}}\right)$, so $(d^\prime,s^\prime,\psi)$ is a valid condition and extends $r=(d^\prime,s^\prime,\varphi^\prime)$ and all $r_l$ for $l\in S=\bigcap_{n\in\New(d^\prime,d)} S_n$, since $n\in\New(d^\prime,d)$ and $l \in S_n$ imply $s^\prime(n)\notin\varphi_l(n)$. 
		\end{proof}


		By using both the standard UF-limits and pseudo-fusions, i.e., by first taking pseudo-fusions and then taking \textit{their} UF-limits, we obtain that $\lebb$ has $\ufi$-limits:
		
		\begin{thm}
			\label{thm_LE_has_UFIC_limits}
			For each long $d$, $Q_{d,s,k}$ is $(D,\frac{1}{|d|})$-lim-linked for any ultrafilter $D$. 
			As a consequence, $\lebb$ has $\ufic$-limits.

		\end{thm}
		\begin{proof}
			
			Given $\bar{q}=\langle q_l:l<\omega\rangle\in (Q_{d,s,k})^\omega$, 
			we define the $(D,\frac{1}{|d|})$-limit condition $q^\infty\coloneq\lim^{D,\frac{1}{|d|}}\bar{q}$ as follows:
			\begin{enumerate}
				\item First, for each $k<\omega$, take the pseudo-fusion $r_k$ of $\bar{q}_k\coloneq\langle q_l:l\in I_k\rangle$ as in Lemma \ref {lem_LE_has_pseudo_fusions}.
				\item Then, take the standard UF-limit $r^\infty$ of the pseudo-fusions $r_k$, i.e., $r^\infty\coloneq\lim^D\langle r_k :k<\omega \rangle$ and define $q^\infty\coloneq r^\infty$, which is possible since each $r_k$ is in $Q_{d,s,k+1}$ by Lemma \ref {lem_LE_has_pseudo_fusions}, so we can take their $D$-limit by Lemma \ref{lem_LE_cUF}.
			\end{enumerate}
			By Lemma \ref {lem_LE_has_pseudo_fusions}, each $r_k$ forces $\dns_k(\bar{q})=\dfrac{|\{l\in I_k:q_l\in \dot{G}\}|}{|I_k|}\geq 1-\dfrac{1}{|d|}$ and by Lemma \ref{lem_LE_cUF}, $q^\infty=r^\infty$ forces $\{k<\omega :r_k\in\dot{G}\}\in\dot{D}^\prime$ where $\dot{D}^\prime$ is the name of an ultrafilter induced by $D$ and $Q_{d,s,k+1}$.
			Thus, $q^\infty$ forces $\left\{k<\omega:\dns_k(\bar{q})\geq 1-\dfrac{1}{|d|}\right\}\in\dot{D}^\prime$.
			
			The latter statement follows from the fact that for any $\varepsilon\in(0,1)_\q$, $\bigcup\{Q_{d,s,k}:d\text{ is long and }\frac{1}{|d|}\leq\varepsilon, s\in\prod_{n<|d|}b(n), k<\omega\}\subseteq\lebb$ is dense and each component $Q_{d,s,k}$ is both centered by Lemma \ref{lem_LE_sigma_centered} and $(D,\varepsilon)$-lim-linked for any ultrafilter $D$ since $\frac{1}{|d|}\leq\varepsilon$.
		\end{proof}
		

		\section{Iteration}\label{sec_separation}

		%
		
		Now we are ready to prove the main theorem, \textit{Cicho\'n's maximum with $\none$ and $\cove$}. We only construct one model including the evasion number $\ee$ and its variants corresponding to Theorem \ref{teo_CM_e_nonE} in Section \ref{sec_intro}, since Theorem \ref{teo_CM_e_nonE} implies Theorem \ref{teo_CM_nonE}.
		We will follow the presentation in \cite[Section 4]{Yam25}, which stems from the original construction of Cicho\'n's maximum in \cite{GKS} and \cite{GKMS}. That is, we first separate the left side of the diagram by performing a fsi 
		and then the right side by the submodel method introduced in \cite{GKMS}.
		We omit some details at several points below since the way to perform the iteration is almost the same as \cite[Section 4]{Yam25}. 

		\begin{dfn}
			Put:
			\begin{itemize}
				\item $\br_0\coloneq\mathbb{C}$, the Cohen forcing.
				\item $\R_1\coloneq\n$ and $\br_1\coloneq\mathbb{A}$, the Amoeba forcing.
				\item $\R_2\coloneq C_\n^\bot$ and $\br_2\coloneq\mathbb{B}$, the random forcing.
				\item $\R_3\coloneq\mathbf{D}$ and $\br_3\coloneq\mathbb{D}$, the Hechler forcing.
				\item $\R_4\coloneq\mathbf{PR}$, $\R_4^*\coloneq\mathbf{BPR}$ and $\br_4\coloneq\mathbb{PR}$.
				\item $\R_5\coloneq C_\mathcal{E}$, $\R_5^g\coloneq \mathbf{PR}_g$ , and $\br_5^g\coloneq\prp_g$ for $g\in(\omega\setminus2)^\omega$.
				\item $\R_6\coloneq C_\m$, $\R_6^*\coloneq\mathbf{LE}$ and $\br_6\coloneq\lebb$.
			\end{itemize}
			$I\coloneq 7$ and $I^+\coloneq I\setminus \{0\}$ are the index sets.
		\end{dfn}
		
		$\br_\ind$ corresponds to the poset which increases $\bb(\R_\ind)$ for each $\ind\in I^+$.
		Also note that $\R_4^*\lq\R_4$, $\R_5\lq \R_5^g$ and $\R_6\lq\R_6^*$ for $g\in(\omega\setminus2)^\omega$.
		
		\begin{ass}
			\label{ass_card_arith}
			\begin{enumerate}
				\item $\lambda_1\leq\cdots\leq\lambda_6$ are uncountable regular cardinals and $\lambda_7\geq\lambda_6$ is a cardinal.
				\item $\lambda_3=\mu_3^+$, $\lambda_4=\mu_4^+$ and $\lambda_5=\mu_5^+$ are successor cardinals and $\mu_3$ is regular.
				\item \label{item_aleph1_inacc}$\kappa<\lambda_\ind$ implies $\kappa^{\aleph_0}<\lambda_\ind$ for all $\ind\in I^+$.
				\item \label{item_ca_3}
				$\lambda_7^{<\lambda_6}=\lambda_7$, hence $\lambda_7^{<\lambda_\ind}=\lambda_7$ for all $\ind\in I^+$.

			\end{enumerate}
		\end{ass}
		
		\begin{dfn}
			Put $\lambda\coloneq\lambda_7$, $S_0\coloneq\lambda$,  $\gamma\coloneq\lambda+\lambda$, the length of the iteration we shall perform. Fix some partition $S_1\cup\cdots\cup S_6=\gamma\setminus S_0$ such that  for each $\ind\in I^+$, $S_i$ is cofinal in $\gamma$.   
			For $\xi<\gamma$, let $\ind(\xi)$ denote the unique $\ind\in I$ such that $\xi\in S_\ind$. 
		\end{dfn}
		
		We additionally assume the following cardinal arithmetic to obtain complete sets of guardrails:
		\begin{ass}
			\label{ass_for_complete_guardrail_E}
			$\lambda_7\leq2^{\mu_3}$.
		\end{ass}
		
		By Lemma \ref{lem_complete} and Assumption \ref{ass_for_complete_guardrail_E}, we have the following (again note that since $|(0,1)_\q|=\aleph_0$, Lemma \ref{lem_complete} also holds for $\ufic$-iterations.):   
		
		\begin{lem}
			For $j=3,4,5$ and $\xi<\gamma$, let $\theta^j_{\xi}\coloneqq\mu_j$ if $\ind(\xi)\leq j$ and $\theta^j_{\xi}\coloneqq\omega$ if $\ind(\xi)> j$.
			Then, there exist complete sets $H$, $H^\prime$ and $H^{\prime\prime}$ of guardrails of length $\gamma=\lambda_7$ for $\lambda_3$-$\Lambda^\mathrm{lim}_\mathrm{uf}$-iteration of size $<\lambda_3$ with $\langle\theta_\xi^3:\xi<\gamma\rangle$, $\lambda_4$-$\Lambda^\mathrm{lim}_\mathrm{cuf}$-iteration of size $<\lambda_4$ with $\langle\theta_\xi^4:\xi<\gamma\rangle$, and $\lambda_5$-$\ufic$-iteration of size $<\lambda_5$ with $\langle\theta_\xi^5:\xi<\gamma\rangle$, respectively.  
		\end{lem}
		
		\begin{con}
			\label{con_P7}
			We shall construct a ccc finite support iteration $\pst^7_\mathrm{pre}$ of length $\gamma$ satisfying the following items:
			\begin{enumerate}
				\item \label{item_P7_1}
				$\pst^7_\mathrm{pre}$ is a $\lambda_3$-$\Lambda^\mathrm{lim}_\mathrm{uf}$-iteration and has $\Lambda^\mathrm{lim}_\mathrm{uf}$-limits on $H$ 
				with the following witnesses:
				\begin{itemize}
					\item $\langle\p_\xi^-:\xi<\gamma\rangle$, the complete subposets witnessing $\Lambda^\mathrm{lim}_\mathrm{uf}$-linkedness.
					\item $\bar{Q}=\langle\dot{Q}_{\xi,\zeta}:\zeta<\theta_\xi^3,\xi<\gamma\rangle$, the $\Lambda^\mathrm{lim}_\mathrm{uf}$-linked components.
					\item $\bar{D}=\langle \dot{D}^h_\xi:\xi\leq\gamma,~h\in H \rangle$, the ultrafilters.
					\item $S^-\coloneq S_0\cup S_1\cup S_2\cup S_3$, the trivial stages and $S^+\coloneq S_4\cup S_5\cup S_6$, the non-trivial stages.
				\end{itemize}
				\item \label{item_P7_2}
				$\pst^7_\mathrm{pre}$ is a $\lambda_4$-$\Lambda^\mathrm{lim}_\mathrm{cuf}$-iteration and has $\Lambda^\mathrm{lim}_\mathrm{cuf}$-limits on $H^\prime$ 
				with the following witnesses:
				\begin{itemize}
					\item $\langle\p_\xi^-:\xi<\gamma\rangle$, the same complete subposets as \eqref{item_P7_1}, which witness $\Lambda^\mathrm{lim}_\mathrm{cuf}$-linkedness as well.
					\item $\bar{R}=\langle\dot{R}_{\xi,\zeta}:\zeta<\theta_\xi^4,\xi<\gamma\rangle$, the $\Lambda^\mathrm{lim}_\mathrm{cuf}$-linked components.
					\item $\bar{E}=\langle \dot{E}^{h^\prime}_\xi:\xi\leq\gamma,~h^\prime\in H^\prime\rangle$, the ultrafilters.
					\item $T^-\coloneq S_0\cup S_1\cup S_2\cup S_3\cup S_4$, the trivial stages and $T^+\coloneq S_5\cup S_6$, the non-trivial stages.
				\end{itemize}
				\item \label{item_P7_3}
				$\pst^7_\mathrm{pre}$ is a $\lambda_5$-$\ufic$-iteration and has $\ufic$-limits on $H^{\prime\prime}$
				with the following witnesses:
				\begin{itemize}
					\item $\langle\p_\xi^-:\xi<\gamma\rangle$, the same complete subposets as \eqref{item_P7_1} and \eqref{item_P7_2}, which witness $\ufic$-linkedness as well.
					\item $\bar{S}=\langle\dot{S}^\varepsilon_{\xi,\zeta}:\varepsilon\in(0,1)_\mathbb{Q},\zeta<\theta_\xi^5,\xi<\gamma\rangle$, the $\ufic$-linked components.
					\item $\bar{F}=\langle \dot{F}^{h^{\prime\prime}}_\xi:\xi\leq\gamma,~h^{\prime\prime}\in H^{\prime\prime}\rangle$, the ultrafilters.
					\item $U^-\coloneq S_0\cup S_1\cup S_2\cup S_3\cup S_4\cup S_5$, the trivial stages and $U^+\coloneq S_6$, the non-trivial stages.
				\end{itemize}
				\item 
				\label{item_N_7}
				For each $\xi\in\gamma\setminus S_0$, $N_\xi\preccurlyeq H_\Theta$ is a submodel where $\Theta$ is a sufficiently large regular cardinal satisfying:
				\begin{enumerate}
					\item $|N_\xi|<\lambda_{\ind(\xi)}$.
					\item \label{item_sigma_closed_E}
					$N_\xi$ is $\sigma$-closed, i.e., $(N_\xi)^\omega\subseteq N_\xi$.
					
					\item 
					\label{item_N_e_E}For any $\ind\in I^+$, $\eta<\gamma$ and for any set of (nice names of) reals $A$ in $V^{\p_\eta}$ of size $<\lambda_\ind$,
					there is some $\xi\in S_\ind$ (above $\eta$) such that $A\subseteq N_\xi$.
					\item If $\ind(\xi)=4$, then $\{\dot{D}^h_\xi:h\in H \}\subseteq N_\xi$.
					\item If $\ind(\xi)=5$, then $\{\dot{D}^h_\xi:h\in H \},\{\dot{E}^{h^\prime}_\xi:h^\prime\in H^\prime \}\subseteq N_\xi$.
					\item If $\ind(\xi)=6$, then $\{\dot{D}^h_\xi:h\in H \},\{\dot{E}^{h^\prime}_\xi:h^\prime\in H^\prime \}, \{ \dot{F}^{h^{\prime\prime}}_\xi:h^{\prime\prime}\in H^{\prime\prime} \}\subseteq N_\xi$.
					
				\end{enumerate}
				
				\item For $\xi\in S_0$, $\p_\xi^-\coloneq\p_\xi$ and for $\xi\in \gamma\setminus S_0$, $\p_\xi^-\coloneq\p_\xi\cap N_\xi$. 
				
				\item For each $\xi<\gamma$,
				$\p^-_\xi\Vdash\qd_\xi\coloneq\br_{\ind(\xi)}$.
				($\br_5$ denotes $\br_5^g$ for some $g\in(\omega\setminus2)^\omega$ and $g$ runs through all $g\in(\omega\setminus 2)^\omega$ by bookkeeping.)
				\item $\bar{Q},\bar{R}$ and $\bar{S}$ are determined in the canonical way: at trivial stages, they consist of singletons and at non-trivial stages, they consist of $\omega$-many $\Lambda^\mathrm{lim}_\mathrm{uf}/\Lambda^\mathrm{lim}_\mathrm{cuf}/$$\ufic$-linked components, respectively (see Table \ref{table_iteration}).


			\end{enumerate}

		\end{con}
		
		\begin{table}[b]
			\caption{Information of each iterand of $\pst^7_\mathrm{pre}$.}\label{table_iteration} 
			\centering
			\begin{tabular}{cccccc|c}
				\hline
				$\ind(\xi)$ &  iterand  & size & $\mu$-UF-linked & $\mu$-cUF & $\mu$-UFI & $\sigma$-centered\\
				
				\hline\hline
				
				1   & $	\mathbb{A}$  & $<\lambda_1$ &  $<\lambda_1$ & $<\lambda_1$ & $<\lambda_1$\\
				2   & $	\mathbb{B}$  & $<\lambda_2$ &  $<\lambda_2$ & $<\lambda_2$ & $<\lambda_2$\\
				3   & $	\mathbb{D}$  & $<\lambda_3$ &  $<\lambda_3$ & $<\lambda_3$ & $<\lambda_3$ &($	\textcolor{black}{\checkmark}$)\\ 
				4   & $	\mathbb{PR}$  & $<\lambda_4$ &  $\omega$ & $<\lambda_4$ & $<\lambda_4$ & $	\textcolor{black}{\checkmark}$\\
				5   & $	\mathbb{PR}_g$  & $<\lambda_5$ &  $\omega$ & $\omega$ & $<\lambda_5$& $	\textcolor{black}{\checkmark}$\\
				6   & $	\mathbb{LE}$  & $<\lambda_6$ &  $\omega$ & $\omega$ & $\omega$& $	\textcolor{black}{\checkmark}$\\
				
				\hline 
			\end{tabular}
			
		\end{table}
		
		We explain why the construction is possible:
		
		\textit{Successor step.} If $\ind(\xi)\leq5$, we are in a similar case to \cite[Construction 4.7]{Yam25} 
		since the iterand is not the new forcing notion $\lebb$ and $\xi$ is a trivial stage for the $\ufic$-iteration, so we do not have to think about $\lebb$ or $\ufic$-limits. Thus, we may assume $\ind(\xi)=6$. Compared to \cite[Construction 4.7]{Yam25}, what we have to additionally check is whether we can extend ultrafilters $\bar{D},\bar{E},\bar{F}$ when the iterand is $\br_{\ind(\xi)}=\lebb$, which is possible since $\lebb$ is both $\sigma$-$\Lambda^\mathrm{lim}_\mathrm{cuf}$-linked and $\sigma$-$\ufic$-linked by Lemma \ref{lem_LE_cUF} and Theorem \ref{thm_LE_has_UFIC_limits}, and by Lemma \ref{lem_uf_const_succ} and \ref{lem_ufi_const_succ}, which guarantee to extend each ultrafilter at successor steps.

		\textit{Limit step.}
		Direct from Lemma \ref{lem_uf_const_lim} and Lemma \ref{lem_ufi_const_lim}, and and Theorem \ref{thm_LE_has_UFIC_limits}, which guarantee to extend each ultrafilter at limit steps (see also Table \ref{table_iteration}). Again note that here is the point where we \textit{crucially need $\sigma$-centeredness} as mentioned in Remark \ref{rem_key_centered}.

		\begin{thm}
			$\pst^7_\mathrm{pre}$ forces for each $\ind\in I^+$,
			$\R_\ind\cong_T C_{[\lambda_7]^{<\lambda_\ind}}\cong_T[\lambda_7]^{<\lambda_\ind}$, in particular, $\bb(\R_\ind)=\lambda_\ind$ and $\dd(\R_\ind)=2^{\aleph_0}=\lambda_7$ (see Figure \ref{fig_p7_pre_mid}, and the same also holds for $\R_4^*$, $\R_5^g$ for $g\in(\omega\setminus2)^\omega$ and $\R_6^*$).
			
		\end{thm}

		\begin{figure}
			\centering
			\begin{tikzpicture}
				\tikzset{
					textnode/.style={text=black}, 
				}
				\tikzset{
					edge/.style={color=black, thin, opacity=0.4}, 
				}
				\newcommand{\w}{2.4}
				\newcommand{\h}{2.0}
				
				\node[textnode] (addN) at (0,  0) {$\addn$};
				\node (t1) [fill=lime, draw, text=black, circle,inner sep=1.0pt] at (-0.25*\w, 0.8*\h) {$\lambda_1$};
				
				\node[textnode] (covN) at (0,  \h*3) {$\covn$};
				\node (t2) [fill=lime, draw, text=black, circle,inner sep=1.0pt] at (0.15*\w, 3.3*\h) {$\lambda_2$};

				\node[textnode] (addM) at (\w,  0) {$\cdot$};
				\node[textnode] (b) at (\w,  1.3*\h) {$\bb$};
				\node (t3) [fill=lime, draw, text=black, circle,inner sep=1.0pt] at (0.68*\w, 1.2*\h) {$\lambda_3$};
				
				\node[textnode] (nonM) at (\w,  \h*3) {$\nonm$};
				\node (t6) [fill=lime, draw, text=black, circle,inner sep=1.0pt] at (1.35*\w, 3.3*\h) {$\lambda_6$};
				
				\node[textnode] (covM) at (\w*2,  0) {$\covm$};
				\node (t7) [fill=lime, draw, text=black, circle,inner sep=1.0pt] at (3.5*\w, 1.5*\h) {{\scalebox{3.0}{$\lambda_7$}}};
				
				\node[textnode] (d) at (\w*2,  1.7*\h) {$\dd$};
				\node[textnode] (cofM) at (\w*2,  \h*3) {$\cdot$};

				\node[textnode] (nonN) at (\w*3,  0) {$\nonn$};
				
				\node[textnode] (cofN) at (\w*3,  \h*3) {$\cofn$};
				
				\node[textnode] (aleph1) at (-\w,  0) {$\aleph_1$};
				\node[textnode] (c) at (\w*4,  \h*3) {$2^{\aleph_0}$};
				
				\node[textnode] (e) at (0.5*\w,  1.7*\h) {$\mathfrak{e}$}; 
				\node[textnode] (estar) at (\w,  1.7*\h) {$\ee^*$};
				\node (t4) [fill=lime, draw, text=black, circle,inner sep=1.0pt] at (0.15*\w, 1.7*\h) {$\lambda_4$};

				\node[textnode] (pr) at (2.5*\w,  1.3*\h) {$\mathfrak{pr}$}; 
				\node[textnode] (prstar) at (\w*2,  1.3*\h) {$\mathfrak{pr}^*$};
				
				\node[textnode] (eubd) at (\w*0.2,  \h*2.1) {$\ee_{ubd}$};
				\node[textnode] (nonE) at (\w*0.75,  \h*2.55) {$\non(\mathcal{E})$};
				\node (t5) [fill=lime, draw, text=black, circle,inner sep=1.0pt] at (0.5*\w, 2.1*\h) {$\lambda_5$};
				
				\node[textnode] (prubd) at (\w*2.8,  \h*0.9) {$\mathfrak{pr}_{ubd}$};
				\node[textnode] (covE) at (\w*2.25,  \h*0.45) {$\cov(\mathcal{E})$};

				\draw[->, edge] (addN) to (covN);
				\draw[->, edge] (addN) to (addM);
				\draw[->, edge] (covN) to (nonM);	
				\draw[->, edge] (addM) to (b);
				\draw[->, edge] (addM) to (covM);
				\draw[->, edge] (nonM) to (cofM);
				\draw[->, edge] (d) to (cofM);
				\draw[->, edge] (b) to (prstar);
				\draw[->, edge] (covM) to (nonN);
				\draw[->, edge] (cofM) to (cofN);
				\draw[->, edge] (nonN) to (cofN);
				\draw[->, edge] (aleph1) to (addN);
				\draw[->, edge] (cofN) to (c);
				
				\draw[->, edge] (e) to (covM);
				\draw[->, edge] (addN) to (e);
				
				\draw[->, edge] (covM) to (prstar);
				\draw[->, edge] (nonM) to (pr);
				\draw[->, edge] (pr) to (cofN);
				
				\draw[->, edge] (e) to (estar);
				\draw[->, edge] (b) to (estar);C
				\draw[->, edge] (estar) to (nonM);
				\draw[->, edge] (estar) to (d);
				\draw[->, edge] (e) to (eubd);
				
				\draw[->, edge] (prstar) to (d);
				\draw[->, edge] (prstar) to (pr);
				
				\draw[->, edge] (prstar) to (pr);
				\draw[->, edge] (prubd) to (pr);
				\
				
				\draw[->, edge] (eubd) to (nonE);
				\draw[->, edge] (addM) to (nonE);
				\draw[->, edge] (nonE) to (nonM);
				
				\draw[->, edge] (covE) to (prubd);
				\draw[->, edge] (covE) to (cofM);
				\draw[->, edge] (covM) to (covE);
				
				\draw[blue,thick] (-0.5*\w,1.5*\h)--(1.5*\w,1.5*\h);
				\draw[blue,thick] (1.5*\w,-0.5*\h)--(1.5*\w,3.5*\h);
				
				\draw[blue,thick] (-0.5*\w,-0.5*\h)--(-0.5*\w,3.5*\h);
				
				\draw[blue,thick] (0.5*\w,-0.5*\h)--(0.5*\w,1.5*\h);
				
				\draw[blue,thick] (-0.1*\w,1.9*\h)--(1.5*\w,1.9*\h);
				\draw[blue,thick] (-0.1*\w,2.7*\h)--(1.5*\w,2.7*\h);
				\draw[blue,thick] (-0.1*\w,1.5*\h)--(-0.1*\w,2.7*\h);
				\draw[blue,thick] (0.5*\w,2.7*\h)--(0.5*\w,3.5*\h);
				

				\draw[->, edge] (nonE) to (nonN);
				\draw[->, edge] (covN) to (covE);

			\end{tikzpicture}
			\caption{Constellation of $\p^7_\mathrm{pre}$ and $\p^7_\mathrm{mid}$.}\label{fig_p7_pre_mid}
		\end{figure}

		\begin{proof}
			Compared with $\p_\mathrm{pre}$ in \cite[Construction 4.7]{Yam25} (or rather the iteration in \cite[Section 4.4]{Yam25} where $\prp_g$ was used), what we have to additionally deal with concerning the new poset $\lebb$ and the new limit notion UFIC are the following:
			\begin{enumerate}
				\item Why $C_{[\lambda_7]^{<\lambda_1}}\lq\R_1$ holds: By Fact \ref{fac_smallness_for_addn_and_covn_and_nonm} and Lemma \ref{lem_LE_sigma_centered}.
				\item Why $C_{[\lambda_7]^{<\lambda_2}}\lq\R_2$ holds: By Fact \ref{fac_smallness_for_addn_and_covn_and_nonm} and Lemma \ref{lem_LE_sigma_centered}.
				\item Why $C_{[\lambda_7]^{<\lambda_3}}\lq\R_3$ holds: By Theorem \ref{thm_uf_limit_keeps_b_small} and Lemma \ref{lem_LE_cUF}.
				\item Why $C_{[\lambda_7]^{<\lambda_4}}\lq\R_4^*$ holds: By Theorem \ref{thm_cUF_keep_estar_small} and Lemma \ref{lem_LE_cUF}.
				\item Why $C_{[\lambda_7]^{<\lambda_5}}\lq\R_5$ holds: By Theorem \ref{thm_UFI_keeps_nonE_small}  and \ref{thm_LE_has_UFIC_limits}.
				\item Why $\R_6^*\lq C_{[\lambda_7]^{<\lambda_6}}$ holds: By 
				Lemma \ref{lem_LE_increases_bbLE}.
			\end{enumerate}
		\end{proof}
		
		As in \cite[Section 4.2]{Yam25}, now we abandon Assumption \ref{ass_for_complete_guardrail_E} and instead assume eventual GCH (which means GCH above some cardinal) to use the \textit{submodel method} to separate the right side, which was introduced in \cite{GKMS}. Note that we will finally need no cardinal arithmetic assumptions in Theorem \ref{thm_p7fin}, but before having the theorem we have to be able to separate the left side even under eventual GCH, to make the final construction argument work (see the proof of Theorem \ref{thm_p7fin}).
		\begin{thm}
			\label{thm_p7_GCH}
			Assume Assumption \ref{ass_card_arith} and for some cardinal $\kappa_0<\lambda_1$, $2^\kappa=\kappa^+$ holds for every $\kappa\geq\kappa_0$.
			Then, there exists a ccc poset $\pst_\mathrm{mid}^7$ which forces for each $\ind\in I^+$,
			$\R_\ind\cong_T C_{[\lambda_7]^{<\lambda_\ind}}\cong_T[\lambda_7]^{<\lambda_\ind}$, in particular, $\bb(\R_\ind)=\lambda_\ind$ and $\dd(\R_\ind)=2^{\aleph_0}=\lambda_7$ (see Figure \ref{fig_p7_pre_mid}, the same also holds for $\R_4^*$, $\R_5^g$ for $g\in(\omega\setminus2)^\omega$ and $\R_6^*$).
		\end{thm}
		\begin{proof}
			See \cite[Section 4.2]{Yam25}.
		\end{proof}
		
				\begin{figure}
			\centering
			\begin{tikzpicture}
				\tikzset{
					textnode/.style={text=black}, 
				}
				\tikzset{
					edge/.style={color=black, thin, opacity=0.4}, 
				}
				\newcommand{\w}{2.4}
				\newcommand{\h}{2.0}
				
				\node[textnode] (addN) at (0,  0) {$\addn$};
				\node (t1) [fill=lime, draw, text=black, circle,inner sep=1.0pt] at (-0.25*\w, 0.8*\h) {$\theta_1$};
				
				\node[textnode] (covN) at (0,  \h*3) {$\covn$};
				\node (t2) [fill=lime, draw, text=black, circle,inner sep=1.0pt] at (0.15*\w, 3.3*\h) {$\theta_2$};

				\node[textnode] (addM) at (\w,  0) {$\cdot$};
				\node[textnode] (b) at (\w,  1.3*\h) {$\bb$};
				\node (t3) [fill=lime, draw, text=black, circle,inner sep=1.0pt] at (0.68*\w, 1.2*\h) {$\theta_3$};
				
				\node[textnode] (nonM) at (\w,  \h*3) {$\nonm$};
				\node (t6) [fill=lime, draw, text=black, circle,inner sep=1.0pt] at (1.35*\w, 3.3*\h) {$\theta_6$};
				
				\node[textnode] (covM) at (\w*2,  0) {$\covm$};
				\node (t7) [fill=lime, draw, text=black, circle,inner sep=1.0pt] at (1.65*\w, -0.3*\h) {$\theta_7$};
				
				\node[textnode] (d) at (\w*2,  1.7*\h) {$\dd$};
				\node (t10) [fill=lime, draw, text=black, circle,inner sep=1.0pt] at (2.32*\w, 1.8*\h) {$\theta_{10}$};
				\node[textnode] (cofM) at (\w*2,  \h*3) {$\cdot$};

				\node[textnode] (nonN) at (\w*3,  0) {$\nonn$};
				\node (t11) [fill=lime, draw, text=black, circle,inner sep=1.0pt] at (2.85*\w, -0.3*\h) {$\theta_{11}$};
				
				\node[textnode] (cofN) at (\w*3,  \h*3) {$\cofn$};
				\node (t12) [fill=lime, draw, text=black, circle,inner sep=1.0pt] at (3.25*\w, 2.2*\h) {$\theta_{12}$};
				
				\node[textnode] (aleph1) at (-\w,  0) {$\aleph_1$};
				\node[textnode] (c) at (\w*4,  \h*3) {$2^{\aleph_0}$};
				\node (t10) [fill=lime, draw, text=black, circle,inner sep=1.0pt] at (3.67*\w, 3.4*\h) {$\theta_\cc$};
				
				\node[textnode] (e) at (0.5*\w,  1.7*\h) {$\mathfrak{e}$}; 
				\node[textnode] (estar) at (\w,  1.7*\h) {$\ee^*$};
				\node (t4) [fill=lime, draw, text=black, circle,inner sep=1.0pt] at (0.15*\w, 1.7*\h) {$\theta_4$};

				\node[textnode] (pr) at (2.5*\w,  1.3*\h) {$\mathfrak{pr}$}; 
				\node[textnode] (prstar) at (\w*2,  1.3*\h) {$\mathfrak{pr}^*$};
				\node (t9) [fill=lime, draw, text=black, circle,inner sep=1.0pt] at (2.85*\w, 1.3*\h) {$\theta_9$};
				
				\node[textnode] (eubd) at (\w*0.2,  \h*2.1) {$\ee_{ubd}$};
				\node[textnode] (nonE) at (\w*0.75,  \h*2.55) {$\non(\mathcal{E})$};
				\node (t5) [fill=lime, draw, text=black, circle,inner sep=1.0pt] at (0.5*\w, 2.1*\h) {$\theta_5$};
				
				\node[textnode] (prubd) at (\w*2.8,  \h*0.9) {$\mathfrak{pr}_{ubd}$};
				\node[textnode] (covE) at (\w*2.25,  \h*0.45) {$\cov(\mathcal{E})$};
				\node (t8) [fill=lime, draw, text=black, circle,inner sep=1.0pt] at (2.47*\w, 0.9*\h) {$\theta_8$};

				\draw[->, edge] (addN) to (covN);
				\draw[->, edge] (addN) to (addM);
				\draw[->, edge] (covN) to (nonM);	
				\draw[->, edge] (addM) to (b);
				\draw[->, edge] (addM) to (covM);
				\draw[->, edge] (nonM) to (cofM);
				\draw[->, edge] (d) to (cofM);
				\draw[->, edge] (b) to (prstar);
				\draw[->, edge] (covM) to (nonN);
				\draw[->, edge] (cofM) to (cofN);
				\draw[->, edge] (nonN) to (cofN);
				\draw[->, edge] (aleph1) to (addN);
				\draw[->, edge] (cofN) to (c);
				
				\draw[->, edge] (e) to (covM);
				\draw[->, edge] (addN) to (e);
				
				\draw[->, edge] (covM) to (prstar);
				\draw[->, edge] (nonM) to (pr);
				\draw[->, edge] (pr) to (cofN);
				
				\draw[->, edge] (e) to (estar);
				\draw[->, edge] (b) to (estar);C
				\draw[->, edge] (estar) to (nonM);
				\draw[->, edge] (estar) to (d);
				\draw[->, edge] (e) to (eubd);
				
				\draw[->, edge] (prstar) to (d);
				\draw[->, edge] (prstar) to (pr);
				
				\draw[->, edge] (prstar) to (pr);
				\draw[->, edge] (prubd) to (pr);
				\
				
				\draw[->, edge] (eubd) to (nonE);
				\draw[->, edge] (addM) to (nonE);
				\draw[->, edge] (nonE) to (nonM);
				
				\draw[->, edge] (covE) to (prubd);
				\draw[->, edge] (covE) to (cofM);
				\draw[->, edge] (covM) to (covE);
				
				\draw[blue,thick] (-0.5*\w,1.5*\h)--(3.5*\w,1.5*\h);
				\draw[blue,thick] (1.5*\w,-0.5*\h)--(1.5*\w,3.5*\h);
				
				\draw[blue,thick] (-0.5*\w,-0.5*\h)--(-0.5*\w,3.5*\h);
				\draw[blue,thick] (3.5*\w,-0.5*\h)--(3.5*\w,3.5*\h);
				
				\draw[blue,thick] (0.5*\w,-0.5*\h)--(0.5*\w,1.5*\h);
				\draw[blue,thick] (2.5*\w,1.5*\h)--(2.5*\w,3.5*\h);
				
				\draw[blue,thick] (-0.1*\w,1.9*\h)--(1.5*\w,1.9*\h);
				\draw[blue,thick] (-0.1*\w,2.7*\h)--(1.5*\w,2.7*\h);
				\draw[blue,thick] (-0.1*\w,1.5*\h)--(-0.1*\w,2.7*\h);
				\draw[blue,thick] (0.5*\w,2.7*\h)--(0.5*\w,3.5*\h);
				
				\draw[blue,thick] (3.1*\w,1.1*\h)--(1.5*\w,1.1*\h);
				\draw[blue,thick] (3.1*\w,0.3*\h)--(1.5*\w,0.3*\h);
				\draw[blue,thick] (3.1*\w,1.5*\h)--(3.1*\w,0.3*\h);
				\draw[blue,thick] (2.5*\w,0.3*\h)--(2.5*\w,-0.5*\h);

				\draw[->, edge] (nonE) to (nonN);
				\draw[->, edge] (covN) to (covE);

			\end{tikzpicture}
			\caption{Constellation of $\p^7_\mathrm{fin}$.}\label{fig_p7fin}
		\end{figure}
		
		Now we separate the right side of the diagram using the submodel method from \cite{GKMS}. The original method assumed eventual GCH, but thanks to an observation from Elliot Glazer (see \cite[Theorem 7.1]{BCM21} and \cite[Remark 5.18]{Mej24Vienna}), we do not need any cardinal arithmetic assumptions for the construction (except for $\theta_\cc^{\aleph_0}=\theta_\cc$) and we can state the theorem in the following way:
		\begin{thm}
			\label{thm_p7fin}
			Let $\aleph_1\leq\theta_1\leq\cdots\leq\theta_{12}$ be regular cardinals and $\theta_\cc$ an infinite cardinal with $\theta_\cc\geq\theta_{12}$ and $\theta_\cc^{\aleph_0}=\theta_\cc$.
			Then, there exists a ccc poset $\pst^7_{\mathrm{fin}}$ which forces $\bb(\R_\ind)=\theta_\ind$ and $\dd(\R_\ind)=\theta_{13-\ind}$ for each $\ind\in I^+$ (the same also holds for $\R_4^*$, $\R_5^g$ for $g\in(\omega\setminus2)^\omega$ and $\R_6^*$) and $2^{\aleph_0}=\theta_\cc$ (see Figure \ref{fig_p7fin}).
		\end{thm}
		
		\begin{proof}
			 In Glazer's method, first we take a regular cardinal $\kappa^*$ and a set $W$ of ordinals coding $V_{\kappa^*}$, and work inside the inner model $L[W]$. Since eventual GCH holds in $L[W]$, we can construct $\pst_\mathrm{mid}^7$ by Theorem \ref{thm_p7fin} (for appropriate $\lambda_i$'s). By using the method from \cite{GKMS}, we can take a submodel $N\preccurlyeq H_\chi$ ($\chi$ is a sufficiently large regular cardinal) such that the poset $\pst^7_{\mathrm{fin}}\coloneqq\pst_\mathrm{mid}^7\cap N$ forces the constellation of Figure \ref{fig_p7fin} and here we use eventual GCH in $L[W]$ to construct such $N$. Finally, we go back to the ground model $V$ and conclude that $\pst^7_{\mathrm{fin}}$ actually forces the same constellation in $V$ as well. For the details, see \cite{GKMS}, \cite[Theorem 7.1]{BCM21} and \cite[Remark 5.18]{Mej24Vienna}.
		\end{proof}
		
		\begin{rem}
			In \cite{BCM21} and \cite{Mej24Vienna}, $\kappa^*$ was supposed to be large enough, but through a private communication with Glazer we found that $\kappa^*$ can be just $\theta_{\mathfrak{c}}^+$.  
		\end{rem}

		\section{Question and Discussion}\label{sec_question}
		\begin{que}
			What about ``closed-UFI-limits''?
		\end{que}
		As in the case of UF-limits, let us define that $(D,\varepsilon)$-lim-linked $Q$ (witnessed by $\lim^{D,\varepsilon}\colon Q\to\p)$ is closed if $\ran(\lim^{D,\varepsilon})\subseteq Q$.
		In the case of $\lebb$, each $(D,\frac{1}{|d|})$-lim-linked component $Q_{d,s,k}$ for a long $d$ is not necessarily closed, since we only know that $\ran(\lim^{D,\varepsilon})\subseteq Q_{d,s,k+1}$, but not $\subseteq Q_{d,s,k}$, though $Q_{d,s,k}$ is closed-UF-lim-linked.
		That is, in the case of $\lebb$, UF-limits are closed but pseudo-fusions are not.
		We have no idea whether there is a forcing notion which has (non-trivial) pseudo-fusions that are closed.
		Also, even if there is a forcing notion with closed-UFI-limits, we are not sure if it has an essentially new application to controlling cardinal invariants.
		\begin{que}
			Does $\lebf_b$ behave differently depending on $b\in\oo$?
		\end{que}
		$\lebf_b$ is similar to $\prbf_g$ in the sense that in both relational systems responses are trying to guess the values of challenges on their infinite set $D\in\ooo$,
		so let us take a look at $\prbf_g$.
		Its bounding number $\bb(\prbf_g)=\ee_g$ is known to depend on $g\in\oo$: Brendle showed in \cite{Bre95} that in the Mathias model $\ee_{ubd}<\ee_2$ holds ($2$ denotes the constant function of the value $2$ on $\omega$), and Spinas proved that if $f,g\in\oo$ are sufficiently different then $\ee_f<\ee_g$ consistently holds (\cite[Theorem 1.9]{Spi98}). However, at this moment, we are not sure if similar arguments work in the case of $\lebf_b$.

		\begin{que}
			Can we separate more numbers in Figure \ref{fig_p7fin}?
		\end{que}
		Focusing on the left side, the remaining parts are the separation of:
		\begin{enumerate}
			\item\label{item_que_sep_one} $\ee$ and $\ee^*$, and 
			\item\label{item_que_sep_two} $\ee_{ubd}$ and $\none$.
		\end{enumerate}
		The first item \eqref{item_que_sep_one} was also considered in \cite[Question 5.3]{Yam25}, and we are still not sure even on the consistency of $\max\{\ee,\bb\}<\ee^*$. On the second item  \eqref{item_que_sep_two}, $\ee_{ubd}<\none$ holds in the Hechler model: Brendle (essentially) showed in \cite{Bre95} that $\ee_{ubd}\leq\ee_{\mathrm{id}}=\aleph_1$ in the Hechler model ($\mathrm{id}$ denotes the identity function on $\omega$), 
		and in that model $\none\geq\adde=\addm=\min\{\bb,\covm\}=2^{\aleph_0}$ holds. 
		However, we have no idea on if we can force $\ee_{ubd}<\none$ in addition to the separation of Figure \ref{fig_p7fin}.
		
		\begin{acknowledgements}
			The author thanks his supervisor J\"{o}rg Brendle for his supportive advice throughout this work. The author is also grateful to Diego~A. Mej\'{\i}a and Miguel~A. Cardona for their helpful comments, particularly on the background of $\none$ and of the limit methods in Section \ref{sec_intro} and on the theory of UFI-limits in Section \ref{sec_UFI}. The author is thankful to the anonymous referee for his/her corrections and suggestions. This work was supported by JST SPRING, Japan Grant Number JPMJSP2148 and JSPS KAKENHI Grant Number JP25KJ1818.
		\end{acknowledgements}

	\end{document}